\documentclass[12pt,reqno]{amsart}
\usepackage{exscale} 
\usepackage{amsmath,amsfonts,amssymb,amsthm,amsxtra,amscd}
\usepackage[all]{xy}
\usepackage{mathrsfs} 
\usepackage{graphicx} 
\numberwithin{equation}{section}

\theoremstyle{plain} 
\newtheorem{theorem}{Theorem}[section] 
\newtheorem{proposition}[theorem]{Proposition} 
\newtheorem{lemma}[theorem]{Lemma} 
\newtheorem{corollary}[theorem]{Corollary}

\theoremstyle{definition} 
\newtheorem{definition}{Definition}[section] 
\newtheorem{example}[theorem]{Example} 
\newtheorem{remark}[theorem]{Remark}

\newcommand{\C}{{\mathbb C}}
\newcommand{\n}{{\mathbb N}} 
\newcommand{\p}{{\mathbb P}}

\newcommand{\z}{{\mathbb Z}} 
\newcommand{\pj}{{{\mathbb P}^1}}
\newcommand{\pii}{{{\mathbb P}^2}}
\newcommand{\piii}{{{\mathbb P}^3}}

\newcommand{\sce}{\mathscr{E}}
\newcommand{\scf}{\mathscr{F}} 

\newcommand{\sco}{\mathscr{O}} 
\newcommand{\sch}{\mathscr{H}}
\newcommand{\sci}{\mathscr{I}}

\newcommand{\scl}{\mathscr{L}}

\newcommand{\tH}{\text{H}} 
\newcommand{\h}{\text{h}}

\newcommand{\izo}{\overset{\sim}{\rightarrow}} 
\newcommand{\Izo}{\overset{\sim}{\longrightarrow}} 
\newcommand{\ra}{\rightarrow} 
\newcommand{\lra}{\longrightarrow} 
\newcommand{\xra}{\xrightarrow}  
\newcommand{\vb}{\, \vert \, } 
\newcommand{\prim}{{\, \prime}} 

\newcommand{\Ker}{\text{Ker}\, }
\newcommand{\Cok}{\text{Coker}\, }

\newcommand{\e}{\varepsilon}

\begin{document}

\title[On the spectrum]{On the spectrum of a stable rank 2 vector bundle on
$\mathbb{P}^3$}

\author[Coand\u{a}]{Iustin~Coand\u{a}} 
\address{Institute of Mathematics ``Simion Stoilow'' of the Romanian Academy, 
         P.O. Box 1-764, 
         RO--014700, Bucharest, Romania} 
\email{Iustin.Coanda@imar.ro} 




\begin{abstract}
The spectrum of a stable rank 2 vector bundle $E$ with $c_1 = 0$ 
on the projective 3-space is a finite
sequence of positive integers $s(0)$, ..., $s(m)$ characterizing the Hilbert function
of the graded $H^1$-module of $E$ in negative degrees.
Hartshorne [Invent. Math. 66 (1982), 165--190] showed that if $s(i) = 1$ for some $i > 0$
then $s(i+1) = 1$, ..., $s(m) = 1$.
We show that if $s(0) = 1$ then $E(1)$ has a global section whose zero scheme is a double
structure on a space curve. We deduce, then, the existence of sequences satisfying
Hartshorne's condition that cannot be the spectrum of any stable 2-bundle. 
This provides a negative answer to a question of
Hartshorne and Rao [J. Math. Kyoto Univ. 31 (1991), 789--806].
\end{abstract}

\maketitle 
\tableofcontents

\section*{Introduction} 

Let $E$ be a stable rank 2 vector bundle on $\piii$ (projective 3-space over an
algebraically closed field $k$ of characteristic 0) with the first Chern class
$c_1 = 0$. Stability means, in this case, that $\tH^0(E) = 0$. Let
$S= k[T_0 , \ldots , T_3]$ be the projective coordinate ring of $\piii$. According to
Horrocks theory, the graded Artinian $S$-module 
$\tH^1_\ast(E) := \bigoplus_{i \in \z}\tH^1(E(i))$ encodes a significant amount of
information about $E$. Let $\h^1_E \colon \z \ra \z$ be the Hilbert function of this
graded module, defined by$\, :$
\[
\h^1_E(i) = \h^1(E(i)) := \dim_k \tH^1(E(i))\, ,\  i \in \z\, . 
\]
As in the case of the Hilbert function of a group of points in $\pii$, one considers,
in order to deal with smaller values, the second difference function of $\h^1_E$.
Recall that for any function $f \colon \z \ra \z$ one defines the difference function
$\Delta f \colon \z \ra \z$ by $(\Delta f)(i) := f(i) - f(i-1)$. Returning to $\h^1_E$,
one puts, for $i \geq 0$,
\[
s(i) := (\Delta^2\h^1_E)(-i-1)\, . 
\]
According to the theorem of Grauert-M\"{u}lich-Spindler \cite{sp}, for a general line
$L \subset \piii$ one has $\tH^0(E_L(-1)) = 0$. If $H \subset \piii$ is a plane
containing $L$ then $\tH^0(E_H(-1)) = 0$. Restricting $E$ to $H$ and then to $L$ one sees
easily that $s(i) \geq 0$, $\forall \, i \geq 0$. Moreover, if
$m := \text{max} \{l \geq 0 \vb \tH^1(E(-l-1)) \neq 0\}$ 
then $s(i) = 0$ for $i > m$. Putting $s(i) := s(-i)$ for $i < 0$, one has$\, :$

\vskip2mm 

\begin{enumerate}
\item[(1)] $\h^1(E(l)) = \h^0({\textstyle \bigoplus}_{i \in \z}\sco_\pj(i+l+1))$,
for $l \leq -1$;
\item[(2)] $\h^2(E(l)) = \h^1({\textstyle \bigoplus}_{i \in \z}\sco_\pj(i+l+1))$,
for $l \geq -3$. 
\end{enumerate}

\vskip2mm 

Indeed, if, for a fixed $i \in \z$, one considers the function $f \colon \z \ra \z$,
$f(l) := \h^0(\sco_\pj(i+l+1))$ then $(\Delta^2f)(-i-1) = 1$ and $(\Delta^2f)(j) = 0$
for $j \neq -i-1$. This proves (1). On the other hand, (2) follows from (1), Serre
duality and the fact that $E \simeq E^\vee$ (because $c_1 = 0$). Since, by Riemann-Roch,
$\h^1(E(-1)) - \h^2(E(-1)) = c_2$ one deduces that $c_2 = \sum_{i \in \z}s(i)$.

The function $s \colon \z \ra \z$ defined above is called the \emph{spectrum} of $E$.
It has been introduced by Barth and Elencwajg \cite{be}, who provided a geometric
interpretation for the vector bundle $K_E := \bigoplus_{i \in \z}s(i)\sco_\pj(i)$ on
$\pj$, and then redefined, algebraically, by Hartshorne \cite{ha}. It is customary
to say, at least in concrete cases, that the spectrum of $E$ is
$(\ldots , 0^{s(0)} , 1^{s(1)} , \ldots , m^{s(m)})$. The spectrum has the following
properties$\, :$

\vskip2mm 

\begin{enumerate}
\item[(3)] (Symmetry) $s(-i) = s(i)$, $\forall \, i \in \z$;
\item[(4)] (Connectedness) $s(i) \geq 1$ for $-m \leq i \leq m$;
\item[(5)] If $s(i) = 1$ for some $i \geq 1$ then $s(j) = 1$ for $i \leq j \leq m$.   
\end{enumerate}

\vskip2mm 

The first two properties have been proved by Barth and Elencwajg (recall that they
have an alternative definition of the spectrum) and the third one by Hartshorne \cite{ha2}.
It is naturally to ask if any function $s \colon \z \ra \z$ satisfying conditions
(3)--(5) is the spectrum of some stable 2-bundle $E$ with $c_1 = 0$.
Hartshorne and Rao \cite{har} showed that this is the case if $c_2 \leq 20$ (for
the case $c_1 = -1$ see Fontes and Jardim \cite{fj}). As a corollary of our results
we shall see that this is no longer true if $c_2 \geq 21$. Before stating our
results we need to recall two other notions. The first one is the degree sequence
of a minimal system of homogeneous generators of $\tH^1_\ast(E)$, namely the function
$\rho \colon \z \ra \z$ defined by$\, :$
\[
\rho(i) := \dim_k(\tH^1(E(i))/S_1\tH^1(E(i-1)))\, ,\  i \in \z\, . 
\]
Secondly, putting $A := \bigoplus_{i \in \z}\rho(i)\sco_\piii(i)$, $E$ admits, according to
Horrocks \cite{ho}, a \emph{monad} of the form$\, :$
\[
A \overset{\alpha}{\lra} B \overset{\beta}{\lra} A^\vee \, , 
\]
with $B$ a direct sum of line bundles, that is, $B = \bigoplus_{i \in \z}b_i\sco_\piii(i)$.
Monad means that $\beta \circ \alpha = 0$, $\beta$ and $\alpha^\vee$ are epimorphisms and
$E \simeq \Ker \beta/\text{Im}\, \alpha$. Moreover, one can assume that the monad is
self-dual, that is, there exists a skew-symmetric isomorphism $\Phi \colon B \izo B^\vee$
such that $\beta = \alpha^\vee \circ \Phi$.

Now, Barth \cite{ba} showed that$\, :$

\vskip2mm

\begin{enumerate}
\item[(6)] $\rho(i) \leq s(-i-1) - 1$, for $-m \leq i < 0$, $\rho(-m-1) = s(m)$ and
$\rho(i) = 0$ for $i < -m-1$. 
\end{enumerate}

\vskip2mm

\noindent
(see, also, Hartshorne and Rao \cite[Prop.~3.1]{har}). We extend this result, in
Theorem~\ref{T:rholeqs-1}, to nonnegative degrees by showing that$\, :$

\vskip2mm

\begin{enumerate}
\item[(7)] $\rho(i) \leq \text{max}(s(-i-1) - 2 , 0)$, for $i \geq 0$. 
\end{enumerate}

\vskip2mm

\noindent
Moreover, we show that if $\rho(i) = s(-i-1) - 1$ for some $i$ with $-m \leq i \leq -2$
then $s(j) = 1$ for $-i \leq j \leq m$. We use, for that, some simplifications of the
arguments in the proof of Hartshorne \cite[Prop.~5.1]{ha2}. Then we express, in
Proposition~\ref{P:bi}, the integers $b_i$ in terms of the functions $s$ and $\rho$.
Using these results and ideas of Rao \cite{r}, we show, in Proposition~\ref{P:s(0)=1},
that$\, :$

\vskip2mm

\begin{enumerate}
\item[(8)] \emph{If} $s(0) = 1$ \emph{then} $E(1)$ \emph{has a global section whose
zero scheme is a double structure} $Y$ \emph{on a locally Cohen-Macaulay curve} $X$
\emph{with} $\tH^1(\sci_X) = 0$ \emph{such that} $\sci_Y$
\emph{is the kernel of an epimorphism} $\sci_X \ra \omega_X(2)$.   
\end{enumerate}

\vskip2mm

In this case, the spectrum of $E$ can be computed by the formula
$s(i) = (\Delta^2\h^0_X)(i)$, $\forall \, i \geq 0$, where $\h^0_X \colon \z \ra \z$
is defined by $\h^0_X(i) := \h^0(\sco_X(i))$. If, moreover, $s(1) = 2$ and $s(2) = 2$
then $\tH^0(\sci_X(2)) \neq 0$, i.e., $X$ is contained in a nonsingular quadric surface,
a quadric cone, the union of two planes or a double plane. Using the results of
Hartshorne and Schlesinger \cite{has} about curves contained in a double plane we 
describe all the functions $\Delta^2\h^0_X$, for $X$ curve with
$\tH^0(\sci_X(2)) \neq 0$ (and $\tH^1(\sci_X) = 0$). This enables us to show that there
are functions $s \colon \z \ra \z$, satisfying (3)--(5) above, that are not the spectrum
of any stable rank 2 vector bundle with $c_1 = 0$. The problem of the determination
of all functions $s \colon \z \ra \z$ with $s(0) = 1$, $s(1) = 2$, $s(2) = 2$ that are
the spectrum of some stable 2-bundle with $c_1 = 0$ seems, however, quite complicated
because one has to identify, among the previously determined functions $\Delta^2\h^0_X$, 
those corresponding to curves admitting an epimorphism $\sci_X \ra \omega_X(2)$.

We finally prove, in Proposition~\ref{P:s(0)=s(1)=2}, a result analogous to (8) above
for the case where $s(0) = 2$ and $s(1) = 2$. This result is more complicated and its
proof uses, among other things, a description of the stable 2-bundles $E$ with
$c_1 = 0$ and $\h^0(E(1)) = 2$. This description is given in
Appendix~\ref{A:h0e(1)=2}. In Appendix~\ref{A:xin2h} we recall the results of
Hartshorne and Schlesinger \cite{has}, complemented by the results of Chiarli, Greco
and Nagel \cite{cgn}, about curves in a double plane.

\vskip2mm

\noindent
{\bf Notation.} (i) By sheaf on a quasi-projective scheme $X$ we always mean
a coherent $\sco_X$-module.

(ii) Unless otherwise explicitly stated, we work only with the closed points
of the schemes under consideration.

(iii) By curve we always mean a locally Cohen-Macaulay, purely 1-dimensional
projective scheme.

(iv) If $X$, $Y$ are closed subschemes of $\p^n$, defined by ideal sheaves $\sci_X$ and
$\sci_Y$, respectively, we write $Y \subseteq X$ if $Y$ is a closed subscheme of $X$,
i.e., if $\sci_X \subseteq \sci_Y$. In this case we put $\sci_{Y , X} := \sci_Y/\sci_X$.
If $\scf$ is a sheaf on $X$ we put $\scf_Y := \scf \otimes_{\sco_X} \sco_Y$. $\scf_Y$
can be identified with the restriction $\scf \vert_Y$ of $\scf$ to $Y$. If $x \in X$
then $\scf_{\{x\}}$ is the reduced stalk $\scf(x)$ of $\scf$ at $x$.

(v) For further notation see the Introduction above.

\section{Preliminaries on monads}\label{S:monads} 

We recall, in this section, the definition and elementary properties of
``splitting monads'' of rank 2 vector bundles on $\piii$, introduced by
Rao \cite{r}, as well as a weaker variant of Rao \cite[Thm.~2.4(ii)]{r}, relating
splitting monads to double structure on (locally Cohen-Macaulay) space curves.

\begin{remark}\label{R:horrocks}
(Horrocks monads)\quad We recall, following Horrocks \cite{ho} and Barth and Hulek
\cite{bh}, the definition and main properties of a (minimal) Horrocks monad of a
vector bundle on $\piii$.
A \emph{Horrocks monad} on $\piii$ is a complex $A^\bullet$ with only three
non-zero terms$\, :$
\[
A^{-1} \xra{d^{-1}} A^0 \xra{d^0} A^1 
\]
that are direct sums of line bundles on $\piii$ and such that $d^0$ and
$d^{-1 \vee}$ are epimorphisms. The \emph{cohomology sheaf} $\mathcal{H}^0(A^\bullet)$ of
the monad is a vector bundle (= locally free sheaf).
The monad is \emph{minimal} if the matrices defining its
differentials contain no entry that is a non-zero constant. By the general
procedure of cancellation of redundant terms in a complex (see, for example,
\cite[Example~A.2]{cho}), every monad is homotopically equivalent to a minimal one.

By Horrocks' method of ``killing cohomology'', every vector bundle $E$ on $\piii$
is the cohomology sheaf of some monad. Indeed, let $A^1$ be a direct sum of line
bundles for which there exists an epimorphism of graded $S$-modules
$e \colon \tH^0_\ast(A^1) \ra \tH^1_\ast(E)$. Consider, also, an epimorphism
$\pi \colon L \ra E^\vee$, with $L$ a direct sum of line bundles. Dualizing $\pi$
one gets an exact sequence$\, :$
\[
0 \lra E \xra{\pi^\vee} L^\vee \lra F \lra 0\, . 
\]
Since the connecting morphism $\tH^0_\ast(F) \ra \tH^1_\ast(E)$ is surjective, $e$
lifts to a morphism $f \colon \tH^0_\ast(A^1) \ra \tH^0_\ast(F)$. The pullback
extension$\, :$
\[
\SelectTips{cm}{12}\xymatrix{0\ar[r] & E\ar[r]\ar @{=}[d] & Q\ar[r]\ar[d] &
A^1\ar[r]\ar[d]^-{\widetilde{f}} & 0\\
0\ar[r] & E\ar[r] & L^\vee\ar[r] & F\ar[r] & 0}
\]
has the property that the connecting morphism $\tH^0_\ast(A^1) \ra \tH^1_\ast(E)$
coincides with $e$, hence $\tH^1_\ast(Q) = 0$. In this case, there exists an
exact sequence$\, :$
\[
0 \lra A^{-1} \xra{d^{-1}} A^0 \lra Q \lra 0\, , 
\]
with $A^{-1}$ and $A^0$ direct sums of line bundles (this follows from ``graded
Serre duality'' as formulated, for example, in \cite[Thm.~1.1]{cho}). The
differential $d^0$ is just the composite morphism $A^0 \ra Q \ra A^1$.

By Barth and Hulek \cite[Prop.~4]{bh}, if $A^\bullet$, $B^\bullet$ are monads with
cohomology sheaves $E$ and $F$, respectively, then the map$\, :$
\[
\text{Hom}_{\text{K}(\piii)}(A^\bullet , B^\bullet) \lra \text{Hom}_{\sco_\piii}(E , F)\, , \
\phi \mapsto \mathcal{H}^0(\phi)\, , 
\]
is bijective. Here $\text{Hom}_{\text{K}(\piii)}$ means ``morphisms of complexes modulo
homotopical equivalence''.

If, moreover, $A^\bullet$ and $B^\bullet$ are minimal and
$\phi \colon A^\bullet \ra B^\bullet$ is a morphism of complexes such that
$\mathcal{H}^0(\phi)$ is an isomorphism then $\phi$ is an isomorphism of complexes.
Indeed, there exists a morphism of complexes $\psi \colon B^\bullet \ra A^\bullet$
such that $\mathcal{H}^0(\psi) = \mathcal{H}^0(\phi)^{-1}$. Then $\psi \phi$ (resp.,
$\phi \psi$) is homotopical equvalent to $\text{id}_{A^\bullet}$ (resp., $\text{id}_{B^\bullet}$).
Since the monads are minimal one deduces, easily, that $\psi \phi$ and $\phi \psi$
are isomorphisms of complexes. It follows that $\phi$ is injective and surjective
hence an isomorphism. 
\end{remark}

\begin{remark}\label{R:selfdual}
(Self-dual monads)\quad Let $E$ be a vector bundle on $\piii$ and $A^\bullet$ a
minimal Horrocks monad of $E$. Assume that $E$ is endowed with a non-degenerate
symplectic form, i.e., with an isomorphism $\sigma \colon E \ra E^\vee$ such that
the composite morphism $E \ra E^{\vee \vee} \overset{\sigma^\vee}{\lra} E^\vee$
equals $- \sigma$. $\sigma$ extends to a morphism of complexes
$\psi \colon A^\bullet \ra A^{\bullet \vee}\, :$
\[
\SelectTips{cm}{12}\xymatrix{A^{-1}\ar[r]^-{d^{-1}}\ar[d]^-{\psi^{-1}} &
A^0\ar[r]^-{d^0}\ar[d]^-{\psi^0} & A^1\ar[d]^-{\psi^1}\\
A^{1 \vee}\ar[r]^-{d^{0 \vee}} & A^{0 \vee}\ar[r]^-{d^{-1 \vee}} & A^{-1 \vee}}
\]
The morphism $\phi := (\psi - \psi^\vee)/2$ extends $\sigma$, too, and has the
additional properties$\, :$
\[
\phi^{0 \vee} = - \phi^0 \text{ and } \phi^{1 \vee} = - \phi^{-1}\, . 
\]
Since $\sigma$ is an isomorphism, $\phi$ is an isomorphism of complexes.

Putting $A := A^{-1}$, $B := A^0$, $\alpha := d^{-1}$, $\beta := \phi^1 \circ d^0$
and $\Phi := \phi^0$, we obtain a minimal monad
$A \overset{\alpha}{\ra} B \overset{\beta}{\ra} A^\vee$ for $E$ and an isomorphism
of monads$\, :$
\[
\SelectTips{cm}{12}\xymatrix{A\ar[r]^-{\alpha}\ar[d]_-{-\text{id}} &
B\ar[r]^-{\beta}\ar[d]^-{\Phi} & A^\vee\ar[d]^-{\text{id}}\\
A\ar[r]^-{\beta^\vee} & B^\vee\ar[r]^-{\alpha^\vee} & A^\vee}
\]
with $\Phi^\vee = - \Phi$. In other words, we obtain a \emph{self-dual monad} 
$A \overset{\alpha}{\ra} B \xra{\alpha^\vee \circ \Phi} A^\vee$.

Notice that if we have an isomorphism $\gamma \colon B^\prime \ra B$, putting
$\alpha^\prime := \gamma^{-1} \circ \alpha$ and
$\Phi^\prime := \gamma^\vee \circ \Phi \circ \gamma$, we obtain a new self-dual
monad and an isomorphism of monads$\, :$
\[
\SelectTips{cm}{12}\xymatrix{A\ar[r]^-{\alpha^\prime}\ar @{=}[d] &
B^\prime\ar[r]^-{\alpha^{\prime \vee} \circ \Phi^\prime}\ar[d]^-{\gamma}_-{\wr} &
A^\vee\ar @{=}[d]\\
A\ar[r]^-{\alpha} & B\ar[r]^-{\alpha^\vee \circ \Phi} & A^\vee}
\]
One gets, analogously, a new self-dual monad from an isomorphism
$\eta \colon B \ra B^\prime$ (with $\alpha^\prime = \eta \circ \alpha$ and
$\Phi^\prime = (\eta^{-1})^\vee \circ \Phi \circ \eta^{-1}$ hence with
$\Phi = \eta^\vee \circ \Phi^\prime \circ \eta$). 
\end{remark}

\begin{lemma}\label{L:semisplitmonad}
Under the hypothesis of Remark~\emph{\ref{R:selfdual}}, $E$ admits a minimal monad
$0 \ra A \overset{\alpha}{\ra} B \overset{\beta}{\ra} A^\vee \ra 0$ with
$B = B_+ \oplus B_0 \oplus B_+^\vee$, where $B_+$ is a direct sum of line bundles of
positive degree and $B_0$ is a trivial bundle, such that $\beta_+ = -\alpha_-^\vee$,
$\beta_- = \alpha_+^\vee$, and $\beta_0 = \alpha_0^\vee \circ \phi$, for some isomorphism
$\phi \colon B_0 \ra B_0^\vee$ with $\phi^\vee = -\phi$. Here $\beta_+$ $($resp., $\alpha_+)$
denotes the restriction $($resp., corestriction$)$ $B_+ \ra A^\vee$ $($resp.,
$A \ra B_+)$ of $\beta$ $($resp., $\alpha)$ etc. 
\end{lemma}

\begin{proof}
Start with a minimal self-dual monad
$0 \ra A \overset{\alpha}{\ra} B \overset{\beta}{\ra} A^\vee \ra 0$ of $E$, with
$\beta = \alpha^\vee \circ \Phi$, for some isomorphism $\Phi \colon B \ra B^\vee$ with
$\Phi^\vee = -\Phi$. Choose a 
decomposition $B = B_+ \oplus B_0 \oplus B_-$, with $\tH^0(B_+^\vee) = 0$, $B_0$
trivial and $\tH^0(B_-) = 0$. Consider, also, the dual decomposition
$B_+^\vee \oplus B_0^\vee \oplus B_-^\vee$ of $B^\vee$. Then
$\Phi \colon B \ra B^\vee$ is represented by a matrix of the form$\, :$
\[
\begin{pmatrix}
0 & 0 & \phi_{13}\\ 0 & \phi_{22} & \phi_{23}\\ \phi_{31} & \phi_{32} & \phi_{33} 
\end{pmatrix}
\]
with $\phi_{ji}^\vee = - \phi_{ij}$. Since $\Phi$ maps $B_+$ isomorphically onto
$B_-^\vee$, $\phi_{31}$ is an isomorphism. Moreover, since $\Phi$ maps
$B_+ \oplus B_0$ isomorphically onto $B_0^\vee \oplus B_-^\vee$, $\phi_{22}$ is an
isomorphism, too.
Consider the automorphisms $\gamma$, $\gamma_1$ of $B$ represented by the
matrices$\, :$
\[
\begin{pmatrix}
\text{id} & 0 & 0\\ 0 & \text{id} & \psi\\ 0 & 0 & \text{id} 
\end{pmatrix}
\text{ and }
\begin{pmatrix}
\text{id} & 0 & \psi_1\\ 0 & \text{id} & 0\\ 0 & 0 & \text{id}
\end{pmatrix} ,   
\]
respectively, for some morphisms $\psi \colon B_- \ra B_0$ and
$\psi_1 \colon B_- \ra B_+$. Then
$\Phi^\prime := \gamma^\vee \circ \Phi \circ \gamma$ is represented by the
matrix$\, :$
\[
\begin{pmatrix}
0 & 0 & \phi_{13}\\ 0 & \phi_{22} & \phi_{23}^\prime\\
\phi_{31} & \phi_{32}^\prime &  \phi_{33}^\prime  
\end{pmatrix}
\]
where $\phi_{23}^\prime = \phi_{23} + \phi_{22}\psi$. Taking
$\psi := - \phi_{22}^{-1}\phi_{23}$ it follows that $\phi_{23}^\prime = 0$ hence
$\phi_{32}^\prime = 0$.

In this case, $\Phi^{\prime \prime} := \gamma_1^\vee \circ \Phi^\prime \circ \gamma_1$
is represented by the matrix$\, :$
\[
\begin{pmatrix}
0 & 0 & \phi_{13}\\ 0 & \phi_{22} & 0\\ \phi_{31} & 0 &  \phi_{33}^{\prime \prime}  
\end{pmatrix}
\]
where
$\phi_{33}^{\prime \prime} = \phi_{33}^\prime + \phi_{31}\psi_1 +\psi_1^\vee\phi_{13}$.
Since $\phi_{13} = - \phi_{31}^\vee$ and $\phi_{33}^{\prime \vee} = - \phi_{33}^\prime$,
if $\psi_1$ is the solution of the equation
$\phi_{31}\psi_1 = - \frac{1}{2}\phi_{33}^\prime$ then $\phi_{33}^{\prime \prime} = 0$.

Put, now, $B^\prime := B_+ \oplus B_0 \oplus B_+^\vee$. $B^{\prime \vee}$ is isomorphic to
$B_+^\vee \oplus B_0^\vee \oplus B_+$. Consider the isomorphism
$\eta := \text{id} \oplus \text{id} \oplus \phi_{13} \colon B \ra B^\prime$. Then
$\Phi^{\prime \prime \prime} := (\eta^{-1})^\vee \circ \Phi^{\prime \prime} \circ \eta^{-1} \colon
B^\prime \ra B^{\prime \vee}$  
is represented by the matrix$\, :$
\[
\Psi = 
\begin{pmatrix}
0 & 0 & \text{id}\\ 0 & \phi_{22} & 0\\ -\text{id} & 0 & 0
\end{pmatrix} 
\] 
because one can easily check that
$\Phi^{\prime \prime} = \eta^\vee \circ \Psi \circ \eta$. 
\end{proof}

\begin{definition}\label{D:semisplitmonad} 
A monad as in the conclusion of Lemma~\ref{L:semisplitmonad} is called a
\emph{semi-splitting monad} of $E$. Notice that, for such a monad, the condition
$\beta \circ \alpha = 0$ is equivalent to$\, :$
\[
\alpha_+^\vee\alpha_- - \alpha_-^\vee\alpha_+ + \alpha_0^\vee \phi \alpha_0 = 0\, . 
\]

One can also define a \emph{reduced semi-splitting monad} for $E$ as follows. 
Consider a decomposition $A = A_{\geq 0} \oplus A_-$, with
$\tH^0(A_{\geq 0}^\vee(-1)) = 0$ and $\tH^0(A_-) = 0$. $\alpha$ maps $A_{\geq 0}$ into
$B_+$. The induced morphism $A_{\geq 0} \ra B_+$ must be a locally split monomorphism
hence its cokernel ${\overline B}_+$ is locally free. Using the diagram$\, :$
\[
\SelectTips{cm}{12}\xymatrix{A_{\geq 0}\ar[r]^-{\text{id}}\ar[d]_{\text{incl}} &
A_{\geq 0}\ar[r]\ar[d]^{\alpha \circ \text{incl}} & 0\ar[d]\\
A\ar[r]^-{\alpha}\ar[d] & B\ar[r]^-{\beta}\ar[d]^{\text{pr} \circ \beta} &
A^\vee\ar[d]^{\text{pr}}\\
0\ar[r] & A_{\geq 0}^\vee\ar[r]^-{\text{id}} & A_{\geq 0}^\vee} 
\]
one deduces that the semi-splitting monad of $E$ is isomorphic, in the derived category
of coherent sheaves on $\piii$, to a monad of the form$\, :$ 
\[
0 \lra A_- \overset{\overline \alpha}{\lra} {\overline B} 
\overset{\overline \beta}{\lra} A_-^\vee \lra 0\, ,  
\]
with ${\overline B} := {\overline B}_+ \oplus B_0 \oplus {\overline B}_+^\vee$, and such
that the components ${\overline \alpha}_+$, ${\overline \alpha}_0$, ${\overline \alpha}_-$
(resp., ${\overline \beta}_+$, ${\overline \beta}_0$, ${\overline \beta}_-$) of
$\overline \alpha$ (resp., $\overline \beta$) are deduced, using the diagram, from the
corresponding components of $\alpha$ (resp., $\beta$). One has an isomorphism of
complexes$\, :$ 
\[
\SelectTips{cm}{12}\xymatrix{A_-\ar[r]^-{\overline \alpha}\ar[d]_-{-\text{id}} &
{{\overline B}_+ \oplus B_0 \oplus {\overline B}_+^\vee}\ar[r]^-{\overline \beta}\ar[d]^-{\Phi} &
A_-^\vee\ar[d]^-{\text{id}}\\
A_-\ar[r]^-{{\overline \beta}^\vee} &
{{\overline B}_+^\vee \oplus B_0^\vee \oplus {\overline B}_+}\ar[r]^-{{\overline \alpha}^\vee}
& A_-^\vee}
\]
where
$\Phi = \left(\begin{smallmatrix} 0 & 0 & \text{id}\\ 0 & \phi & 0\\ -\text{id} & 0 & 0
\end{smallmatrix}\right)$,
that is, ${\overline \beta}_+ = -{\overline \alpha}_-^\vee$,
${\overline \beta}_- = {\overline \alpha}_+^\vee$ and
${\overline \beta}_0 = {\overline \alpha}_0^\vee \phi$. Notice that
${\overline \alpha}_0$ can be identified with $\alpha_0$ because $\alpha_0$ vanishes on
$A_{\geq 0}$ and ${\overline \beta}_0$ can be identified with $\beta_0$ because $\beta_0$
maps $B_0$ into $A_-^\vee$.

The preceeding monad is called a \emph{reduced semi-splitting monad} of $E$. 
\end{definition}   

The statement of the next result is inspired by a result of
Rao \cite[Thm.~2.4(ii)]{r}.

\begin{lemma}\label{L:splittingdouble}
Let $E$ be a rank $2$ vector bundle on $\piii$ with $c_1 = 0$ and consider 
a reduced semi-splitting monad of $E$ as in
Definition~\emph{\ref{D:semisplitmonad}}. Put
$c := c_1({\overline B}_+^\vee) - c_1(A_-)$. Let $B_0^\prime$ be a totally isotropic 
trivial subbundle of $B_0$ with respect to $\phi \colon B_0 \izo B_0^\vee$,
and let $\beta_0^\prime \colon B_0^\prime \ra A_-^\vee$ be the restriction of
$\beta_0$.

Assume, now, that the degeneracy locus of 
$({\overline \beta}_+ , \beta_0^\prime) \colon {\overline B}_+ \oplus B_0^\prime \ra A_-^\vee$ 
has \emph{(}the expected\emph{)} codimension $2$,
hence that it is a locally Cohen-Macaulay space curve $X$. Then $E(c)$ has
a global section whose zero scheme is a double structure $Y$ on $X$ such that the
ideal sheaf $\sci_Y$ is the kernel of an epimorphism
$\sci_X \ra \omega_X(4 - 2c)$. 
\end{lemma} 

\begin{proof}
Recall that \emph{totally isotropic} means that the rank of $B_0^\prime$ is half the rank of
$B_0$ and that the composite map
$B_0^\prime\hookrightarrow B_0\overset{\phi}{\ra} B_0^\vee\twoheadrightarrow B_0^{\prime\vee}$ 
is the zero map. In particular, the rank of ${\overline B}_+ \oplus B_0^\prime$ equals
$\text{rk}\, A_-^\vee + 1$ because $E$ has rank $2$.
Let $M^\bullet$ denote the reduced semi-splitting monad of $E$ and let $K^\bullet$ be its
subcomplex$\, :$
\[
0 \lra {\overline B}_+ \oplus B_0^\prime
\xra{({\overline \beta}_+ , \beta_0^\prime)} A_-^\vee\, . 
\]
The isomorphism $M^\bullet \izo M^{\bullet \vee}$ at the end of Definition~\ref{D:semisplitmonad}
induces an isomorphism $M^\bullet/K^\bullet \izo K^{\bullet \vee}$. 
Now, using our hypothesis, the Eagon-Northcott complex associated to
$({\overline \beta}_+ , \beta_0^\prime)^\vee$ gives an exact sequence$\, :$ 
\begin{equation}\label{E:ix(c)} 
0 \lra A_- \xra{({\overline \beta}_+ , \beta_0^\prime)^\vee}
B_0^{\prime \vee} \oplus {\overline B}_+^\vee \lra \sci_X(c) \lra 0\, . 
\end{equation}
Dualizing it, one gets the exact sequence$\, :$
\[
0 \lra \sco_\piii(-c) \lra {\overline B}_+ \oplus B_0^\prime
\xra{({\overline \beta}_+ , \beta_0^\prime)} 
A_-^\vee \lra \omega_X(4-c) \lra 0\, . 
\]
The short exact sequence of complexes
$0 \ra K^\bullet \ra M^\bullet \ra M^\bullet/K^\bullet \ra 0$ produces, now, an
exact sequence$\, :$ 
\[
0 \lra \sco_\piii(-c) \lra E \lra \sci_X(c) \lra \omega_X(4-c) \lra 0\, . \qedhere 
\]
\end{proof}

\begin{corollary}\label{C:splittingdouble}
Under the hypothesis of Lemma~\emph{\ref{L:splittingdouble}} one has
${\fam0 h}^0(\sci_X(c)) = {\fam0 h}^0(B_0^{\prim \vee})$,
${\fam0 H}^0(\sci_X(c-1)) = 0$ and ${\fam0 H}^1(\sci_X(l)) = 0$ for $l \leq c-1$. 
\end{corollary}

\begin{proof}
One uses the exact sequence \eqref{E:ix(c)} in the above proof of Lemma~\ref{L:splittingdouble}
and the dual of the defining exact sequence$\, :$
\[
0 \lra A_{\geq 0} \lra B_+ \lra {\overline B}_+ \lra 0\, . \qedhere 
\]
\end{proof}

\section{Spectra and monads}\label{S:spectra}

We recall, in this section, Hartshorne's algebraic definition \cite[Thm.~7.1]{ha} 
of the spectrum of a stable rank 2 vector bundle $E$ on $\piii$ with $c_1 = 0$ (the
original definition of Barth and Elencwajg is geometric). We then show that the
spectrum imposes certain (not necessarily optimal) restrictions on the shape of the
minimal monad of the bundle. Our results improve and extends previous results of
Barth \cite{ba} and Hartshorne and Rao \cite[Prop.~3.1]{har} who limited themselves
to the case where the graded $S$-module $\tH^1_\ast(E)$ is generated in negative
degrees, i.e., to the case where the right hand term of the monad is a direct sum
of positive line bundles.

Hartshorne's definition of the spectrum is based on the following technical result
(see \cite[Thm.~5.3]{ha}). The difficult part of this theorem is assertion (a) for
$i = -1\, ;$ the remaining assertions follow easily from the Bilinear Map Lemma
\cite[Lemma~5.1]{ha} by decreasing induction on $i$. A different proof of the former
assertion, based on Beilinson's theorem, can be found in \cite[Thm.~A.1]{co}.

\begin{theorem}\label{T:ninh1f}
Let $F$ be a stable rank $2$ vector bundle on $\pii$ with $c_1 = 0$, let $R$ be the
homogeneous coordinate ring of $\pii$, and let $N$ be a graded $R$-submodule of
${\fam0 H}^1_\ast(F)$. Put $n_i := \dim_k N_i$, $i \in \z$. Then$\, :$

\emph{(a)} $n_i > n_{i-1}$ if $N_{i-1} \neq 0$, $\forall \, i \leq -1$, except when
$i = -1$, $N_{-2} = {\fam0 H}^1(F(-2))$ and $N_{-1} = {\fam0 H}^1(F(-1))$ $($recall
that ${\fam0 h}^1(F(-2)) = c_2 = {\fam0 h}^1(F(-1)))$.

\emph{(b)} If $N_{i-1} \neq 0$ and $n_i - n_{i-1} = 1$, for some $i \leq -2$, then
there exists a non-zero linear form $\ell \in R_1$ such that $\ell N_j = 0$ in
${\fam0 H}^1(F(j+1))$, $\forall \, j \leq i$. 
\end{theorem}

\begin{corollary}\label{C:h1fraq}
Under the hypothesis of Theorem~\emph{\ref{T:ninh1f}}, consider the quotient
$R$-module $Q := {\fam0 H}^1_\ast(F)/N$ and put
$q_i := \dim_k Q_i = \dim_k({\fam0 H}^1(F(i))/N_i)$, $i \in \z$. Then$\, :$

\emph{(a)} $q_i > q_{i+1}$ if $Q_{i+1} \neq 0$, $\forall \, i \geq -1$.

\emph{(b)} If $Q_{i+1} \neq 0$ and $q_i - q_{i+1} = 1$, for some $i \geq -1$, then
there exists a non-zero linear form $\ell \in R_1$ such that multiplication by
$\ell \colon Q_{j-1} \ra Q_j$ is the zero map, $\forall \, j \geq i$. 
\end{corollary}

\begin{proof}
$Q_j^\vee$ embeds into $\tH^1(F(j))^\vee$ which, by Serre duality and the fact that
$F^\vee \simeq F$, is isomorphic to $\tH^1(F(-j-3))$. One can apply, now,
Theorem~\ref{T:ninh1f} to the graded $R$-submodule $Q^\vee(3)$ of $\tH^1_\ast(F)$.
Recall that $Q^\vee : = \bigoplus_{j \in \z}(Q_{-j})^\vee$.   
\end{proof}

As a supplement to Theorem~\ref{T:ninh1f}, we notice that the linear form $\ell$
from item (b) of the statement of that theorem is \emph{unique}
(up to multiplication by a non-zero constant), as the following easy lemma shows.

\begin{lemma}\label{L:xi}
Let $F$ be a stable rank $2$ vector bundle on $\pii$ with $c_1 = 0$, let
$\ell ,\, \ell^\prime \in R_1$ be two linearly independent linear forms, and let
$\xi$ be an element of ${\fam0 H}^1(F(j))$, for some $j \leq -2$. If $\ell \xi = 0$
and $\ell^\prime \xi = 0$ in ${\fam0 H}^1(F(j+1))$ then $\xi = 0$. 
\end{lemma}

\begin{proof}
Let $x \in \pii$ be the point of equations $\ell = \ell^\prime = 0$. Tensorizing by
$F(j+1)$ the exact sequence
$0 \ra \sco_\pii(-1) \ra 2\sco_\pii \ra \sci_{\{x\}}(1) \ra 0$ one gets an exact
sequence$\, :$
\[
\tH^0(\sci_{\{x\}} \otimes F(j+2)) \lra \tH^1(F(j))
\xra{(\ell\, ,\, \ell^\prime)^{\text{t}}} 2\tH^1(F(j+1))\, . 
\]
One uses, now, the fact that $\tH^0(F(j+2)) = 0$ because $j \leq -2$. 
\end{proof}

We recall, now, Hartshorne's definition of the spectrum of a stable rank 2 vector
bundle on $\piii$.

\begin{definition}\label{D:spectrum}
Let $E$ be a stable rank 2 vector bundle on $\piii$ with $c_1 = 0$ and
$c_2  \geq 2$. Let $H \subset \piii$ be a plane of equation $h = 0$ such that
$E_H$ is stable. Put$\, :$
\begin{gather*}
N := \text{Im}(\tH^1_\ast(E) \lra \tH^1_\ast(E_H)) \simeq
\Cok (\tH^1_\ast(E(-1)) \overset{h}{\lra} \tH^1_\ast(E))\, ,\\
Q := \Cok (\tH^1_\ast(E) \lra \tH^1_\ast(E_H)) \simeq
\Ker (\tH^2_\ast(E(-1)) \overset{h}{\lra} \tH^2_\ast(E))\, ,
\end{gather*}
$n_i := \dim_k N_i$ and $q_i := \dim_k Q_i$, $i \in \z$. Notice that$\, :$
\[
n_i = \h^1(E(i)) - \h^1(E(i-1)) \text{ if } i \leq 0\, ,\
q_i = \h^2(E(i-1)) - \h^2(E(i)) \text{ if } i \geq -3\, , 
\]
hence these integers do not depend on the choice of the plane $H$ (provided that
$E_H$ is stable). Notice, also, that $Q = \tH^1_\ast(E_H)/N$. Moreover, by Serre
duality and the fact that $E^\vee \simeq E$, one has
$\tH^1_\ast(E)^\vee \simeq (\tH^2_\ast(E(-1))(-3)$ hence $N^\vee \simeq Q(-3)$ hence
$Q \simeq N^\vee(3)$.

One defines, now, the \emph{spectrum} of $E$ as the function $s \colon \z \ra \z$
given by the formula$\, :$
\[
s(i) :=
\begin{cases}
n_{-i-1} - n_{-i-2}\, , &\text{if $i \geq 0$;}\\
q_{-i-2} - q_{-i-1}\, , &\text{if $i \leq -1$.}  
\end{cases}  
\]
Since $q_{-2} + n_{-2} = \h^1(E_H(-2)) = c_2 = \h^1(E_H(-1)) = q_{-1} + n_{-1}$ it
follows that the relation $s(i) = q_{-i-2} - q_{-i-1}$ is valid for $i = 0$, too.
Consider the nonnegative integer$\, :$
\[
m := \text{max}\, \{l \geq 0 \, \vert \, \tH^1(E(-l-1)) \neq 0\} =
\text{max}\, \{l \geq 0 \, \vert \, \tH^2(E(l-3)) \neq 0\}\, . 
\]

\noindent
{\bf Claim.}\quad (a) $s(i) = 0$ \emph{for} $i < -m$ \emph{and for} $i > m$.

(b) $s(i) \geq 1$ \emph{for} $-m \leq i \leq m$.

\vskip2mm

\noindent
\emph{Proof of the claim}.  
(a) One has $N_i = 0$ for $i \leq -m-2$ and $Q_i = 0$ for $i \geq m-1$, by the
definition of $m$.

(b) One has $N_{-m-1} \simeq \tH^1(E(-m-1))$ and $Q_{m-2} \simeq \tH^2(E(m-3))$.
If $m = 0$ the assertion follows from the fact that $\tH^1(E(-1)) \neq 0$, which
can be deduced from Riemann-Roch. Assume, now, that $m \geq 1$. Since
$N_{-m-1} \neq 0$ (resp., $Q_{m-2} \neq 0$), Theorem~\ref{T:ninh1f}(a) (resp.,
Corollary~\ref{C:h1fraq}(a)) implies that $n_{-m-1} < n_{-m} < \cdots < n_{-2}$
(resp., $q_{-1} > q_0 > \cdots > q_{m-2}$). Moreover, since $Q_{-1} \neq 0$ it follows
that $N_{-1} \neq \tH^1(E_H(-1))$ hence, by Theorem~\ref{T:ninh1f}(a),
$n_{-2} < n_{-1}$. \hfill $\Diamond$ 

\vskip2mm 

Considering, now, the vector bundle $K_E := \bigoplus_{i \in \z}s(i)\sco_\pj(i)$ on
$\pj$ one sees that the spectrum characterizes half of the Hilbert functions of the
intermediate cohomology modules of $E$ according to the relations$\, :$
\begin{enumerate}
\item[(i)] $\h^1(E(l)) = \h^0(K_E(l+1))$, for $l \leq -1\, ;$
\item[(ii)] $\h^2(E(l)) = \h^1(K_E(l+1))$, for $l \geq -3$.  
\end{enumerate}
One also has the relation$\, :$
\begin{enumerate}
\item[(iii)] ${\textstyle \sum_{i \in \z}}s(i) = q_{-1} + n_{-1} = \h^1(E_H(-1)) = c_2$.
\end{enumerate}
Moreover, since $Q \simeq N^\vee(3)$, it follows that the spectrum is symmetric, 
that is$\, :$
\begin{enumerate}
\item[(iv)] $s(-i) = q_{i-2} - q_{i-1} = n_{-i-1} - n_{-i-2} = s(i)$,
$\forall \, i \geq 0$. 
\end{enumerate}
\end{definition}

\begin{lemma}\label{L:Delta2h0}
Using the notation from Definition~\emph{\ref{D:spectrum}}, let $c \geq 1$ be an integer such
that $E(c)$ has a global section whose zero scheme is a double structure $Y$ on a locally
Cohen-Macaulay curve $X$ such that $\sci_Y$ is the kernel of an epimorphism
$\sci_X \ra \omega_X(4-2c)$. Assume, also, that ${\fam0 H}^i(\sci_X(l)) = 0$ for $l \leq c-1$,
$i = 0,\, 1$. Then
\[
s(i) = (\Delta^2{\fam0 h}^0_X)(i+c-1)\, , \  \forall \, i \geq 0\, ,
\]
where ${\fam0 h}^0_X \colon \z \ra \z$ is defined by
${\fam0 h}^0_X(i) := {\fam0 h}^0(\sco_X(i))$. 
\end{lemma}

\begin{proof}
Using the exact sequences$\, :$
\[
0 \ra \sco_\piii(-c) \ra E \ra \sci_Y(c) \ra 0\, ,\
0 \ra \sci_Y \ra \sci_X \ra \omega_X(4-2c) \ra 0\, , 
\]
one deduces that, $\forall \, i \geq 0\, :$
\[
\h^1(E(-i-1)) = \h^1(\sci_Y(c-i-1)) = \h^0(\omega_X(3-c-i)) =
\h^1(\sco_X(i+c-3))\, . 
\]
It follows that $n_{-i-1} = \h^1(\sco_X(i+c-3)) - \h^1(\sco_X(i+c-2))$ hence$\, :$
\[
s(i) := n_{-i-1} - n_{-i-2} = (\Delta^2\h^1_X)(i+c-1)\, , \forall \, i \geq 0\, . 
\]
But, by Riemann-Roch, $\Delta^2(\h^0_X - \h^1_X) = 0$. 
\end{proof}  

The next two lemmas will be used in the proof of the main result of this section,
Theorem~\ref{T:rholeqs-1} below.

\begin{lemma}\label{L:unstable}
With the notation from Definition~\emph{\ref{D:spectrum}}, assume that $m \geq 1$. 
Let $H \subset \piii$ be a plane such that $E_H$ is stable.  
If there exists a non-zero linear form
$\ell \in {\fam0 H}^0(\sco_H(1))$ such that $\ell N_{-m-1} = 0$ in
${\fam0 H}^1(E(-m))$ then there exists a plane $H_0 \subset \piii$ such that
${\fam0 H}^0(E_{H_0}(-m)) \neq 0$. 
\end{lemma}

\begin{proof}
Consider a linear form $\lambda \in \tH^0(\sco_\piii(1))$ such that
$\lambda \vert_H = \ell$. The hypothesis is equivalent to the fact that$\, :$ 
\[
\lambda \tH^1(E(-m-1)) \subseteq h \tH^1(E(-m-1)) \text{ in } \tH^1(E(-m))\, . 
\]
Multiplication by $h \colon \tH^1(E(-m-1)) \ra \tH^1(E(-m))$ maps $\tH^1(E(-m-1))$
isomorphically onto $h\tH^1(E(-m-1))$ (because $\tH^0(E_H(-m)) = 0$). It follows that
multiplication by $\lambda \colon \tH^1(E(-m-1)) \ra h\tH^1(E(-m-1))$ can be
viewed as an automorphism of the vector space $\tH^1(E(-m-1))$. Let $c \in k$ be an
eigenvalue of this automorphism. Putting $h_0 := \lambda - ch$, it follows that
multiplication by $h_0 \colon \tH^1(E(-m-1)) \ra \tH^1(E(-m))$ is not injective. If
$H_0$ is the plane of equation $h_0 = 0$ then $\tH^0(E_{H_0}(-m)) \neq 0$. 
\end{proof}

\begin{lemma}\label{L:sj=1}
With the notation from Definition~\emph{\ref{D:spectrum}}, assume that $m \geq 2$
and consider an integer $i$ with $-m-1 \leq i \leq -2$. Assume that, for any general
plane $H \subset \piii$ such that $E_H$ is stable, there exists a linear form
$\ell \in {\fam0 H}^0(\sco_H(1))$ such that $\ell N_j = 0$ in
${\fam0 H}^1(E_H(j+1))$, $\forall \, j \leq i$. Then $s(j) = 1$ for
$-i-1 \leq j \leq m$. 
\end{lemma}

\begin{proof}
By Lemma~\ref{L:unstable}, there exists a plane $H_0 \subset \piii$, of equation
$h_0 = 0$, such that $\tH^0(E_{H_0}(-m)) \neq 0$. Let $s$ be a non-zero global section
of $E_{H_0}$. Since $\tH^0(E_{H_0}(-m-1)) = 0$ (because $\tH^1(E(-m-2)) = 0$) it
follows that the zero scheme $Z$ of $s$ is a zero dimensional subscheme of $H_0$.
One deduces an exact sequence
$0 \ra \sco_{H_0}(m) \overset{s}{\ra} E_{H_0} \ra \sci_{Z , H_0}(-m) \ra 0$.

Since any general plane satisfies the hypothesis of the lemma, it follows that
there exists such a plane $H$ that does not intersect $Z$. Consider the line
$L_0 := H \cap H_0$ and let $\ell_0 := h_0 \vert_H \in \tH^0(\sco_H(1))$.

\vskip2mm

\noindent
{\bf Claim.}\quad $\ell_0 N_j = 0$ \emph{in} $\tH^1(E_H(j+1))$,
$\forall \, j \leq i$.

\vskip2mm

\noindent
\emph{Proof of the claim}. One knows, by hypothesis, that there exists a
non-zero linear form $\ell \in \tH^0(\sco_H(1))$ such that $\ell N_j = 0$ in
$\tH^1(E_H(j+1))$, $\forall \, j \leq i$. Now, $h_0$ annihilates a non-zero element
of $\tH^1(E(-m-1))$ hence $\ell_0$ annihilates a non-zero element of $N_{-m-1}$.
One deduces, from Lemma~\ref{L:xi}, that $\ell_0$ and $\ell$ are linearly
dependent. \hfill $\Diamond$

\vskip2mm

Now, using the exact sequence
$0 \ra E_H(-1) \overset{\ell_0}{\lra} E_H \ra E_{L_0} \ra 0$, one deduces, for
$j \leq -1$, an exact sequence$\, :$
\[
0 \lra \tH^0(E_{L_0}(j+1)) \overset{\partial}{\lra} \tH^1(E_H(j))
\overset{\ell_0}{\lra} \tH^1(E_H(j+1))\, . 
\]
Notice that, since $L_0$ does not intersect $Z$ in $H_0$, one has
$E_{L_0} \simeq \sco_{L_0}(m) \oplus \sco_{L_0}(-m)$. Since, by the Claim, $\ell_0$
annihilates $N_j$ for $j \leq i$, $N_j$ can be identified to a subspace of
$\tH^0(E_{L_0}(j+1)) = \tH^0(\sco_{L_0}(m + j + 1))$ for $j \leq i$. Since
$0 \neq N_{-m-1} \subseteq \tH^0(\sco_{L_0})$, one deduces that
$N_j = \tH^0(\sco_{L_0}(m + j + 1))$ for $j \leq i$. It follows that$\, :$
\[
s(-j-1) = n_j - n_{j-1} = 1 \text{ for } -m-1 \leq j \leq i 
\]
hence $s(j) = 1$ for $-i-1 \leq j \leq m$. 
\end{proof}

Before stating our results relating the shape of the minimal monad to the spectrum,
we introduce one more definition due to Barth \cite{ba} (see, also
Hartshorne and Rao \cite[\S 3]{har}).

\begin{definition}\label{D:rho}
With the notation from Definition~\ref{D:spectrum}, let $\rho(i)$ denote the
number of minimal generators of degree $i$ of the graded $S$-module
$\tH^1_\ast(E)$. In other words,
\[
\rho(i) := \dim_k(\tH^1(E(i))/S_1\tH^1(E(i-1)))\, . 
\]
Since $N \simeq \tH^1_\ast(E)/h\tH^1_\ast(E(-1))$ it follows that$\, :$
\begin{equation}\label{E:rhoi}
\rho(i) = \dim_k(N_i/R_1N_{i-1}) = \dim_k(\tH^1(E_H(i))/R_1N_{i-1}) -
\dim_k(\tH^1(E_H(i))/N_i)\, ,  
\end{equation}
where $R : = S/Sh$. Recalling the definition of the integer $m$, one has$\, :$ 

\vskip2mm

\noindent
{\bf Claim 1.}\quad (a) $\rho(i) = 0$ \emph{for} $i < -m-1$ \emph{and} 
$\rho(-m-1) = \h^1(E(-m-1)) = s(m)$.

(b) $\rho(i) = 0$ \emph{for} $i \geq m$.

\vskip2mm

\noindent
\emph{Proof of Claim} 1. (a) is obvious. As for (b), since
$\tH^2(E(m-2)) = 0$ and $\tH^3(E(m-3)) \simeq \tH^0(E(-m-1))^\vee = 0$, the
Castelnuovo-Mumford lemma (in its slightly more general variant stated in
\cite[Lemma~1.21]{acm}) implies that $\tH^1_\ast(E)$ is generated in degrees
$\leq m-1$ hence $\rho(i) = 0$ for $i \geq m$. \hfill $\Diamond$

\vskip2mm 

According to Lemma~\ref{L:semisplitmonad}, $E$ admits a (minimal) semi-splitting monad 
\[
0 \lra A \overset{\alpha}{\lra} B_+ \oplus B_0 \oplus B_+^\vee \overset{\beta}{\lra}
A^\vee \lra 0\, ,  
\]
with $A := \bigoplus_{i \in \z}\rho(i)\sco_\piii(i)$, $B_+ = \bigoplus_{i > 0}b_i\sco_\piii(i)$ and
$B_0 = b_0\sco_\piii$.  

\vskip2mm

\noindent
{\bf Claim 2.}\quad $b_i = 0$ \emph{for} $i > m$. 

\vskip2mm

\noindent
\emph{Proof of Claim} 2. Since $\tH^0(E) = 0$, the sequence
$0 \ra \tH^0(A(l)) \ra \tH^0(B_+(l)) \ra \tH^0(A^\vee(l))$ is exact for $l < 0$. Since
$\tH^0(A(-m)) = 0$ and $\tH^0(A^\vee(-m-2)) = 0$ (by Claim 1) 
it follows that $\tH^0(B_+(-m-2)) = 0$ hence
$b_i = 0$ for $i \geq m+2$. Moreover, the map $\tH^0(B_+(-m-1)) \ra \tH^0(A^\vee(-m-1))$ is
injective. On the other hand, this map is the zero map because the monad is minimal. One
deduces that $b_{m+1} = 0$. \hfill $\Diamond$

\vskip2mm 

Let us finally recall, from Definition~\ref{D:semisplitmonad}, that $E$ admits also 
a reduced semi-splitting monad$\, :$
\[
0 \lra A_- \overset{\overline \alpha}{\lra} 
{\overline B}_+ \oplus B_0 \oplus {\overline B}_+^\vee
\overset{\overline \beta}{\lra} A_-^\vee \lra 0\, ,
\]
where $A_- := \bigoplus_{i < 0}\rho(i)\sco_\piii(i)$ and ${\overline B}_+$ is defined
by an exact sequence$\, :$
\[
0 \lra A_{\geq 0} \lra B_+ \lra {\overline B}_+ \lra 0\, , \text{ where }
A_{\geq 0} := {\textstyle \bigoplus_{i \geq 0}}\rho(i)\sco_\piii(i)\, . 
\]
\end{definition}

Assertion (a) of the next result is due to Barth \cite{ba}, at least for $i < 0$,
(see, also, Hartshorne and Rao \cite[Prop.~3.1]{har}) while assertion (b) refines
a result of Hartshorne \cite[Prop.~5.1]{ha2} (which is, actually,
Corollary~\ref{C:unstablesj1}). Assertion (c) is, however, novel.

\begin{theorem}\label{T:rholeqs-1}
With the notation from Definition~\emph{\ref{D:spectrum}} and
Definition~\emph{\ref{D:rho}}, one has$\, :$

\emph{(a)} $\rho(i) \leq s(-i-1) - 1$, for $-m \leq i \leq m-1$.

\emph{(b)} If $\rho(i) = s(-i-1) - 1$, for some $i$ with $-m \leq i \leq -2$, then
$s(j) = 1$ for $-i \leq j \leq m$ and there is a plane $H_0 \subset \piii$ such
that ${\fam0 H}^0(E_{H_0}(-m)) \neq 0$.

\emph{(c)} $\rho(i) \leq {\fam0 max}(s(-i-1) - 2 , 0)$, for $i \geq 0$. 
\end{theorem}

\begin{proof}
(a) Let $H \subset \piii$ be a plane of equation $h = 0$ such that $E_H$ is stable.  
Consider, for $-m \leq i \leq m-1$, the graded submodule $N^{(i)}$ of 
$\tH^1_\ast(E_H)$ defined by $N^{(i)}_j := N_j$ for $j \neq i$ and
$N^{(i)}_i := R_1N_{i-1}$, where $R := S/Sh$. Consider, also, the quotient module
$Q^{(i)} := \tH^1_\ast(E_H)/N^{(i)}$ and put $n^{(i)}_j := \dim_k N^{(i)}_j$,
$q^{(i)}_j := \dim_k Q^{(i)}_j$.

Assume, firstly, that $-m \leq i \leq -1$. One has $N^{(i)}_{-m-1} \neq 0$ because
$N^{(i)}_{-m-1} = N_{-m-1}$ and $N^{(i)}_{-1} \neq \tH^1(E_H(-1))$ because
$N^{(i)}_{-1} \subseteq N_{-1} \neq \tH^1(E_H(-1))$ (see Definition~\ref{D:spectrum}).
Theorem~\ref{T:ninh1f}(a) implies, now, that $n^{(i)}_{-m-1} < n^{(i)}_{-m} < \cdots < n^{(i)}_{-1}$.
In particular, $n_{i-1} = n^{(i)}_{i-1} < n^{(i)}_i$ hence$\, :$
\[
\rho(i) = n_i - n^{(i)}_i < n_i - n_{i-1} = s(-i-1)\, . 
\]

Assume, now, that $0 \leq i \leq m-1$. One has $Q^{(i)}_{m-2} \neq 0$ because
$Q_{m-2}$ is a quotient of $Q^{(i)}_{m-2}$. Corollary~\ref{C:h1fraq}(a) implies that
$q^{(i)}_{-1} > q^{(i)}_0 > \cdots > q^{(i)}_{m-2}$. Moreover, one also has
$q^{(i)}_{m-2} > q^{(i)}_{m-1}$ because either $Q^{(i)}_{m-1} \neq 0$ and one applies
Corollary~\ref{C:h1fraq}(a) or $Q^{(i)}_{m-1} = 0$ and one uses the fact that
$Q^{(i)}_{m-2} \neq 0$. In particular, $q_{i-1} = q^{(i)}_{i-1} > q^{(i)}_i$ hence$\, :$
\[
\rho(i) = q^{(i)}_i - q_i < q_{i-1} - q_i = s(-i-1)\, . 
\]

(b) The proof of (a) shows that $n^{(i)}_i - n^{(i)}_{i-1} = 1$.
Theorem~\ref{T:ninh1f}(b) implies that there exists a non-zero linear form
$\ell \in \tH^0(\sco_H(1))$ such that $\ell N^{(i)}_j = 0$ in $\tH^1(E_H(j+1))$,
$\forall \, j \leq i$. In particular, $\ell N_j = 0$, $\forall \, j \leq i-1$.
Since $H$ is any plane for which $E_H$ is stable, Lemma~\ref{L:sj=1} implies that
$s(j) = 1$ for $-i \leq j \leq m$. Moreover, Lemma~\ref{L:unstable} implies that
there exists a plane $H_0 \subset \piii$ such that $\tH^0(E_{H_0}(-m)) \neq 0$.

(c) Assume that $\rho(i) = s(-i-1) - 1$ for some $i$ with $0 \leq i \leq m-1$.
We want to show that, in this case, $\rho(i) = 0$. The proof of (a) shows that
$q^{(i)}_{i-1} - q^{(i)}_i = 1$.

We assume, firstly, that $0 \leq i \leq m-2$ or $i = m-1$ and
$Q^{(m-1)}_{m-1} \neq 0$. In this case $Q^{(i)}_i \neq 0$ and
Corollary~\ref{C:h1fraq}(b) implies that there exists a non-zero linear form
$\ell \in \tH^0(\sco_H(1))$ such that multiplication by
$\ell \colon Q^{(i)}_{j-1} \ra Q^{(i)}_j$ is the zero map, $\forall \, j \geq i-1$.
Since $Q$ is a quotient of $Q^{(i)}$, the same is true for the multiplication maps
$\ell \colon Q_{j-1} \ra Q_j$, $j \geq i-1$.

Using the fact that $Q \simeq N^\vee(3)$, one deduces that the multiplication by
$\ell \colon N_l \ra N_{l+1}$ is the zero map, $\forall \, l \leq -i-2$. Since $H$ is
any plane for which $E_H$ is stable, Lemma~\ref{L:sj=1} implies that $s(j) = 1$, 
for $i+1 \leq j \leq m$. In particular, $s(i+1) = 1$ hence $s(-i-1) = 1$ hence
$\rho(i) = 0$.

Assume, finally, that $i = m-1$ and $Q^{(m-1)}_{m-1} = 0$. In this case
$\rho(m-1) = q^{(m-1)}_{m-1} - q_{m-1} = 0$. 
\end{proof}

\begin{corollary}\label{C:unstablesj1}
Besides the properties emphasized in Definition~\emph{\ref{D:spectrum}}, the
spectrum $s$ of a stable rank $2$ vector bundle $E$ on $\piii$ with $c_1 = 0$
also satisfies the following condition$\, :$
\begin{enumerate}
\item[(v)] If $s(i) = 1$ for some $i$ with $1 \leq i \leq m$ then $s(j) = 1$ for
$i \leq j \leq m$. 
\end{enumerate}
If, moreover, $1 \leq i \leq m-1$ then there is a plane $H_0 \subset \piii$ such that
${\fam0 H}^0(E_{H_0}(-m)) \neq 0$. 
\end{corollary}

\begin{proof}
We can assume, of course, that $1 \leq i \leq m-1$. Put $l := -i-1$, such that
$-l-1 = i$. One has $-m \leq l \leq -2$. Since $\rho(l) \leq s(-l-1) - 1$ and
since $s(-l-1) = 1$ one gets that $\rho(l) = s(-l-1) - 1$. The assertions from the
corollary follow, now, from Theorem~\ref{T:rholeqs-1}(b). 
\end{proof}  

\begin{remark}\label{R:Q2}
According to Definition~\ref{D:spectrum} and Corollary~\ref{C:unstablesj1} (which
is, actually, Hartshorne's result \cite[Prop.~5.1]{ha2}) the spectrum of a stable
rank 2 vector bundle $E$ on $\piii$ with $c_1 = 0$ is a function
$s \colon \z \ra \n$, with finite support, satisfying the following
conditions$\, :$
\begin{enumerate}
\item[(S1)] (Symmetry)\quad $s(-i) = s(i)$, $\forall \, i \in \z\, ;$
\item[(S2)] (Connectednes)\quad $s(0) \geq 1$ and if $s(i) \geq 1$, for some
$i \geq 1$, then $s(j) \geq 1$, for $0 \leq j \leq i\, ;$
\item[(S3)] If $s(i) = 1$ for some $i \geq 1$ then $s(j) \leq 1$,
$\forall \, j \geq i$. 
\end{enumerate}
Hartshorne and Rao ask, in question (Q2) at the end of their paper \cite{har}, 
if any function with the above properties is the spectrum of some stable rank 2
vector bundle $E$ on $\piii$ with $c_1 = 0$. 
\end{remark}

\begin{proposition}\label{P:bi}
Let $E$ be a stable rank $2$ vector bundle on $\piii$ with $c_1 = 0$. Consider, as
in the final part of Definition~\emph{\ref{D:rho}}, the minimal semi-splitting monad and
the reduced semi-splitting monad of $E$. Then, with the notation from that definition,
one has$\, :$

\emph{(a)} $b_i = 2s(i) - (s(i-1) - \rho(-i)) - (s(i+1) - \rho(i))$, for $i \geq 1$.

\emph{(b)} $b_0/2 = s(0) - (s(1) - \rho(0)) + 1$.

\emph{(c)} $c_1({\overline B}_+^\vee)-c_1(A_-) = c_1(A_-^\vee)-c_1({\overline B}_+) = s(0)$. 
\end{proposition}

\begin{proof}
(a) Consider a flag $L \subset H \subset \piii$, where $L$ is a line and $H$ a
plane. Since $\tH^0(E(-1)) = 0$, it follows that$\, :$
\[
\h^1(E(-l)) = \h^0(A^\vee(-l)) - \h^0(B_+(-l)) + \h^0(A(-l))\, ,\ \forall \,
l \geq 1\, . 
\]
One deduces that, for $i \geq 1\, :$
\[
n_{-i} = \h^1(E(-i)) - \h^1(E(-i-1)) = \h^0(A^\vee_H(-i)) - \h^0((B_+\vert_H)(-i)) +
\h^0(A_H(-i))\, , 
\]
hence, for $i \geq 0\, :$
\[
s(i) = n_{-i-1} - n_{-i-2} = \h^0(A^\vee_L(-i-1)) - \h^0((B_+\vert_L)(-i-1)) +
\h^0(A_L(-i-1))
\]
hence, still for $i \geq 0$,
\[
s(i) - s(i+1) = {\textstyle \sum_{j \geq i+1}}\rho(-j) -
{\textstyle \sum_{j \geq i+1}}b_j + {\textstyle \sum_{j \geq i+1}}\rho(j)\, . 
\] 
Finally, for $i \geq 1\, :$
\[
(s(i-1) - s(i)) - (s(i) - s(i+1)) = \rho(-i) - b_i + \rho(i)\, . 
\]

(b) Taking $\sum_{i \geq 1}$ of the last relations in the proof of (a), one gets$\, :$
\[
s(0) - s(1) = {\textstyle \sum_{i \geq 1}}(\rho(-i) - b_i + \rho(i))\, . 
\]
hence $\sum_{i \geq 1}b_i = \sum_{i \neq 0}\rho(i) -s(0) +s(1)$. 
On the other hand, $2\text{rk}\, B_+ + \text{rk}\, B_0 = 2\text{rk}\, A + 2$ hence$\, :$
\[
{\textstyle \sum_{i \geq 1}}b_i + b_0/2 = {\textstyle \sum_{i \in \z}}\rho(i) + 1\, .
\]
Substracting the former relation from the latter one, one gets the relation from the
statement.

(c) One of the relations established in the proof of (a) says that$\, :$
\begin{gather*}
s(0) = \h^0(A^\vee_L(-1)) - \h^0((B_+\vert_L)(-1)) + \h^0(A_L(-1)) =\\
\h^0((A_-^\vee \vert_L)(-1)) - \h^0((B_+\vert_L)(-1)) + \h^0((A_{\geq 0} \vert_L)(-1))\, .  
\end{gather*}
One uses, now, the fact that if $G$ is a vector bundle on $\pj$ that is a direct
sum of line bundles of non-negative degrees then $\h^0(G(-1)) = c_1(G)$ and the
fact that if $C$ is a vector bundle on $\piii$ then $c_1(C) = c_1(C_L)$. 
\end{proof}  

\begin{corollary}\label{C:si=2}
With the notation from Definitions~\emph{\ref{D:spectrum}} and
\emph{\ref{D:rho}}, one has$\, :$

\emph{(a)} If $s(0) = 1$ and $s(1) \geq 2$ then $\rho(-1) = 0$,
$\rho(0) = s(1) - 2$ and $b_0 = 0$. Moreover,
$b_1 = 2s(1) - 1 - (s(2) - \rho(1))$.

\emph{(b)} If $s(0) = 2$, $s(1) = 2$ and $s(2) \geq 2$ then $\rho(-2) = 0$, $\rho(-1) \leq 1$,
$\rho(0) = 0$ and $b_0 = 2$. Moreover, $b_1 = \rho(-1) + 2 - (s(2) - \rho(1)) \leq \rho(-1)$. 

\emph{(c)} If $s(i) = 2$ and $s(i+1) \geq 2$, for some $i$ with
$2 \leq i \leq m-1$, then $\rho(i-1) = 0$, $\rho(i) = s(i+1) - 2$,
$\rho(-i) = s(i-1) - 2$ and $b_i = 0$. 
\end{corollary}

\begin{proof}
(a) Theorem~\ref{T:rholeqs-1}(a) implies that $\rho(-1) = 0$. Since
$s(1) = s(-1)$, Theorem~\ref{T:rholeqs-1}(c) implies that $\rho(0) \leq s(1) - 2$.
Now, the formula in Proposition~\ref{P:bi}(b) implies that $b_0 = 0$ and that
$\rho(0) = s(1) - 2$. The formula for $b_1$ follows, now, from
Proposition~\ref{P:bi}(a).

(b) Since $s(2) \geq 2$, Theorem~\ref{T:rholeqs-1}(a),(b) implies that $\rho(-2) = 0$. It also
follows, from Theorem~\ref{T:rholeqs-1}(a) (resp., Theorem~\ref{T:rholeqs-1}(c)) that
$\rho(-1) \leq 1$ (resp., $\rho(0) = 0$). Proposition~\ref{P:bi}(b) implies, now, that $b_0 = 2$
and that $b_1 = \rho(-1) + 2 - (s(2) - \rho(1))$. One deduces, from Theorem~\ref{T:rholeqs-1}(c),
that $s(2) - \rho(1) \geq 2$ hence $b_1 \leq \rho(-1)$. 

(c) Since $s(-i) = s(i) = 2$, Theorem~\ref{T:rholeqs-1}(c) implies that
$\rho(i-1) = 0$ and Theorem~\ref{T:rholeqs-1}(b) implies that
$\rho(-i) \leq s(i-1) - 2$. Since $s(i+1) = s(-i-1)$, Theorem~\ref{T:rholeqs-1}(c)
also implies that $\rho(i) \leq s(i+1) - 2$. The formula for $b_i$ in
Proposition~\ref{P:bi}(a) shows, now, that $b_i = 0$, $\rho(-i) = s(i-1) - 2$ and
$\rho(i) = s(i+1) - 2$. 
\end{proof}

\section{Spectra and double structures on space curves, I}\label{S:doublei}

We draw, in this section and in the next one, some consequences of the general results from the
previous sections. Using the notation from Definition~\ref{D:spectrum} and Definition~\ref{D:rho},
we show that if $E$ is a stable rank 2 vector bundle on $\piii$ with $s(0) \leq 2$ and 
$s(1) = s(2) = 2$ then, essentially, $E(s(0))$ has a global section whose zero scheme is
a double structure on a (locally Cohen-Macaulay) space curve $X$ with
$\tH^0(\sci_X(2)) \neq 0$. We use this fact and the description of the space curves
contained in a double plane given by Hartshorne and Schlesinger \cite{has} to show that
the question (Q2) concerning the spectrum of a stable rank 2 bundle on $\piii$ formulated
at the end of the paper of Hartshorne and Rao \cite{har} has a negative answer. We actually
emphasize a further condition that the spectrum must satisfy but, unfortunately, that condition
is far from being sufficient to characterize the sequences of nonnegative
integers that can be the spectrum of a stable rank 2 vector bundle on $\piii$. 

We gather, in the next remark, some easy observations that will be needed in the sequel.

\begin{remark}\label{R:c1scf}
(a) If $\scf$ is a torsion sheaf on $\p^n$ then$\, :$
\[
c_1(\scf) = {\textstyle \sum}_X(\text{length}\, \scf_\xi)\text{deg}\, X\, , 
\]
where the sum is indexed by the 1-codimensional irreducible components of the support of
$\scf$, $\xi$ denotes the generic point of $X$ and $\text{length}\, \scf_\xi$ is the
length of the Artinian $\sco_{\p^n , \xi}$-module $\scf_\xi$.

\emph{Indeed}, let $\Sigma$ be the support of $\scf$, endowed with the structure of
reduced closed subscheme of $\p^n$. The successive quotients of the filtration
$\scf \supseteq \sci_\Sigma\scf \supseteq \sci_\Sigma^2\scf \supseteq \cdots$ are
$\sco_\Sigma$-modules. Consequently, we can assume that $\scf$ is an $\sco_\Sigma$-module.
If $X$ is an irreducible component of $\Sigma$, there exists a non-empty open subset
$U$ of $X$, intersecting none of the other components of $\Sigma$, such that
$\scf \vert_U$ is a locally free $\sco_U$-module. Then $\text{length}\, \scf_\xi$
equals the rank of $\scf \vert_U$. One can restrict, now, $\scf$ to a general line
$L \subset \p^n$.

\vskip2mm

(b) Let $a_1 \leq a_2 \leq \cdots \leq a_m$ be a sequence of integers and let $\scf$ be
a coherent quotient of $\bigoplus_{i = 1}^m\sco_{\p^n}(a_i)$. If $\scf$ has rank
$r \geq 1$ then $c_1(\scf/\scf_{\text{tors}}) \geq a_1 + \cdots + a_r$.

\emph{Indeed}, choose a (closed) point $x \in \p^n$ such that $\scf/\scf_{\text{tors}}$
is locally free in a neighbourhood of $x$. The reduced stalk $\e(x)$ of the
epimorphism $\e \colon \bigoplus_{i=1}^m\sco_{\p^n}(a_i) \ra \scf/\scf_{\text{tors}}$
is surjective. There exist indices $1 \leq i_1 < \cdots < i_r \leq m$ such that
$\e(x)$ induces an isomorphism
$\bigoplus_{j=1}^r\sco_{\p^n}(a_{i_j})(x) \izo (\scf/\scf_{\text{tors}})(x)$. The restriction of
$\e$ to $\bigoplus_{j = 1}^r\sco_{\p^n}(a_{i_j})$ is an isomorphism in a neighbourhood
of $x$ hence this restriction is a monomorphism (globally) and its cokernel is a torsion sheaf. 
It follows that
$c_1(\scf/\scf_{\text{tors}}) \geq c_1(\bigoplus_{j = 1}^r\sco_{\p^n}(a_{i_j}))$ (one takes into account
assertion (a) above). 

\vskip2mm

(c) Let $\phi \colon F \ra G$ be a morphism of locally free sheaves on $\p^n$. Then$\, :$
\begin{gather*} 
c_1(\Ker \phi) - c_1(\Ker \phi^\vee) = c_1((\Cok \phi)_{\text{tors}}) +  
2c_1((\Cok \phi)/(\Cok \phi)_{\text{tors}}) +\\ c_1(F) - c_1(G)\, . 
\end{gather*}

\emph{Indeed}, $c_1(\Ker \phi) = c_1(F) - c_1(G) + c_1(\Cok \phi)$. On the other hand,
$\Ker \phi^\vee \simeq (\Cok \phi)^\vee$. If $\scf$ is a coherent sheaf on $\p^n$ then,
dualizing the exact sequence
$0 \ra \scf_{\text{tors}} \ra \scf \ra \scf/\scf_{\text{tors}} \ra 0$, one gets that
$(\scf/\scf_{\text{tors}})^\vee \izo \scf^\vee$. One deduces that
$c_1(\Ker \phi^\vee) = -c_1(\Cok \phi /(\Cok \phi)_{\text{tors}})$. 
\end{remark}

\begin{proposition}\label{P:s(0)=1}
Let $E$ be a stable rank $2$ vector bundle on $\piii$ with $c_1 = 0$. Consider the
spectrum of $E$ \emph{(}see Definition~\emph{\ref{D:spectrum}}\emph{)} and its minimal
\emph{(}resp., reduced\emph{)} semi-splitting monad
\emph{(}see Definition~\emph{\ref{D:rho}}\emph{)}.
Assume that $s(0) = 1$ and $s(1) \geq 2$ hence that, by Corollary~\emph{\ref{C:si=2}(a)},
$b_0 = 0$ and $\rho(-1) = 0$.

Then the degeneracy scheme of
${\overline \beta}_+ \colon {\overline B}_+ \ra A_-^\vee$ is a locally Cohen-Macaulay
curve $X$ with ${\fam0 H}^1(\sci_X) = 0$ and $E(1)$ has a global section whose zero scheme is a
double structure $Y$ on $X$ such that the ideal sheaf $\sci_Y$ is the kernel of an epimorphism
$\sci_X \ra \omega_X(2)$. 
\end{proposition}

\begin{proof}
The reduced semi-splitting monad of $E$ has the shape
$0 \ra A_- \ra {\overline B}_+ \oplus {\overline B}_+^\vee \ra A_-^\vee \ra 0$. By
Proposition~\ref{P:bi}(c), one has $c_1({\overline B}_+^\vee) - c_1(A_-) = s(0) = 1$. Taking into 
account Lemma~\ref{L:splittingdouble} (and its corollary), one sees that it suffices to show that the
degeneracy locus of ${\overline \beta}_+$ has codimension $\geq 2$ in $\piii$.
We use the exact sequence$\, :$
\[
0 \ra \Ker {\overline \beta}_+^\vee \ra \Ker {\overline \beta}_+ \ra E \ra
\Cok {\overline \beta}_+^\vee \ra \Cok {\overline \beta}_+ \ra 0\, , 
\]
induced by the exact sequence of complexes
$0 \ra K^\bullet \ra M^\bullet \ra M^\bullet/K^\bullet \ra 0$, where $M^\bullet$ denotes the
reduced semi-splitting monad and $K^\bullet$ its subcompex
$0 \ra {\overline B}_+ \overset{{\overline \beta}_+}{\lra} A_-^\vee$ 
(see the proof of Lemma~\ref{L:splittingdouble}). Since
$\Ker {\overline \beta}_+^\vee \simeq (\Cok {\overline \beta}_+)^\vee$ and since
$\text{rk}\, {\overline B}_+ = \text{rk}\, A_-^\vee + 1$ (because $E$ has rank 2), it
follows that
$\text{rk}\, \Ker {\overline \beta}_+ = \text{rk}\, \Ker {\overline \beta}_+^\vee + 1$.
The quotient 
$\Ker {\overline \beta}_+/\Ker {\overline \beta}_+^\vee$ embeds into $E$ hence it is a
torsion free sheaf of rank 1. Since $E$ is stable with $c_1 = 0$ one deduces that$\, :$
\[
c_1(\Ker {\overline \beta}_+) - c_1(\Ker {\overline \beta}_+^\vee) \leq -1\, . 
\]
On the other hand, applying the formula from Remark~\ref{R:c1scf}(c) to
$\phi = {\overline \beta}_+$ and recalling that
$c_1({\overline B}_+^\vee) - c_1(A_-) = s(0) = 1$, one gets that$\, :$
\[
c_1(\Ker {\overline \beta}_+) - c_1(\Ker {\overline \beta}_+^\vee) =
c_1((\Cok {\overline \beta}_+)_{\text{tors}}) +
2c_1(\Cok {\overline \beta}_+/(\Cok {\overline \beta}_+)_{\text{tors}}) -1\, . 
\]
If $\Cok {\overline \beta}_+$ is not a torsion sheaf 
then, by Remark~\ref{R:c1scf}(b),
\[
c_1(\Cok {\overline \beta}_+/(\Cok {\overline \beta}_+)_{\text{tors}}) \geq 1 
\]
(it is, actually, even $\geq 2$ because $\rho(-1) = 0$). On the other hand, by
Remark~\ref{R:c1scf}(a), $c_1((\Cok {\overline \beta}_+)_{\text{tors}}) \geq 0$. It follows
that $\Cok {\overline \beta}_+$ is a torsion sheaf and that one has 
$c_1(\Cok {\overline \beta}_+) = 0$ hence the support of $\Cok {\overline \beta}_+$ has
codimension at least 2 in $\piii$. 
\end{proof}

\begin{corollary}\label{C:s(0)=1}
Under the hypothesis of Proposition~\emph{\ref{P:s(0)=1}}, if, moreover,
$s(1) = s(2) = 2$ then ${\fam0 H}^0(\sci_X(2)) \neq 0$ $($and ${\fam0 H}^1(\sci_X(l)) = 0$
for $l \leq 2)$. 
\end{corollary}

\begin{proof}
Since $s(2) = 2$, Theorem~\ref{T:rholeqs-1}(c) implies that $\rho(1) = 0$. One deduces,
from Corollary~\ref{C:si=2}(a), that $\rho(-1) = 0$, $\rho(0) = 0$, $b_0 = 0$ and
$b_1 = 1$. As we saw in the proof of Lemma~\ref{L:splittingdouble}, one has an exact
sequence$\, :$
\[
0 \lra A_- \overset{{\overline \beta}_+^\vee}{\lra} {\overline B}_+^\vee \lra
\sci_X(1) \lra 0 
\]
(recall that $c := c_1({\overline B}_+^\vee) - c_1(A_-) = s(0) = 1$). Since $\rho(-1) = 0$
it follows that $\h^0(\sci_X(2)) = \h^0({\overline B}_+^\vee(1))$. Using the dual of the defining
exact sequence$\, :$
\[
0 \lra A_{\geq 0} \lra B_+ \lra {\overline B}_+ \lra 0 
\]
and taking into account that $\tH^0(A_{\geq 0}^\vee(1)) = 0$ (because $\rho(0) = \rho(1) = 0$)  
one gets that $\h^0({\overline B}_+^\vee(1)) = b_1 = 1$. Moreover, 
$\tH^1(\sci_X(l)) \simeq \tH^1({\overline B}_+^\vee(l-1)) = 0$ for $l \leq 2$. 
\end{proof}

We describe now (see the below examples) the second difference functions $\Delta^2\h^0_X$ for
the curves $X$ with $\tH^0(\sci_X(2)) \neq 0$. 
One needs for that only a convenient description of
curves contained in the union of two planes (see Hartshorne \cite{ha3} and
Remark~\ref{R:xinh0h1} below)
and of curves contained in a double plane (due to Hartshorne and Schlesinger \cite{has}
and complemented by Chiarli, Greco and Nagel \cite{cgn}; their results are recalled in
Appendix~\ref{A:xin2h}).

We state, firstly, some observations concerning the computation of $\Delta^2\h^0_X$,
for $X$ space curve, that will be useful in the examples below.

\begin{remark}\label{R:Delta2h0}
(a) If $\tH^0(\sco_X(-1)) = 0$ then $(\Delta^2\h^0_X)(i) = 0$ for $i < 0$ and
$(\Delta^2\h^0_X)(0) = \h^0(\sco_X)$.

(b) Let $e$ be the \emph{speciality index} of $X$, defined by $\tH^1(\sco_X(e)) \neq 0$
and $\tH^1(\sco_X(e+1)) = 0$. Then $(\Delta^2\h^0_X)(i) = 0$ for $i > e+1$ and
$(\Delta^2\h^0_X)(e+1) = \h^1(\sco_X(e))$.

\emph{Indeed}, as we noticed in the proof of Lemma~\ref{L:Delta2h0},
$\Delta^2\h^0_X = \Delta^2\h^1_X$ (by Riemann-Roch).

(c) Assume that $X$ is arithmetically Cohen-Macaulay and consider a resolution
$0 \ra A \ra B \ra \sci_X \ra 0$, with $A$, $B$ direct sums of line bundles. Then, for any
line $L \subset \piii\, :$
\[
(\Delta^2\h^0_X)(i) = \h^0(\sco_L(i)) - \h^0(B_L(i)) + \h^0(A_L(i))\, ,\
\forall \, i \in \z\, . 
\]
Notice, in this context, that if $a < b$ are integers and $f \colon \z \ra \z$ is defined by
$f(i) : = \h^0(\sco_L(i-a)) - \h^0(\sco_L(i-b))$ then $f(i) = 0$ for $i < a$, while for
$i \geq a\, :$ 
\[
f(i) =
\begin{cases}
i - a + 1, &\text{if $a \leq i < b$;}\\
b-a,       &\text{if $i \geq b$}  
\end{cases}
\text{ and }
(\Delta f)(i) =
\begin{cases}
1, &\text{if $a \leq i < b$;}\\
0, &\text{if $i \geq b$.}  
\end{cases}
\]

(d) If $X$ is contained in a nonsingular surface $\Sigma \subset \piii$ then, for every
curve $X^\prime$ in the complete linear system $\vert \sco_\Sigma(X) \vert$, one has
$\h^0_X = \h^0_{X^\prime}$.

\emph{Indeed}, one has an exact sequence
$0 \ra \sco_\Sigma(-X) \ra \sco_\Sigma \ra \sco_X \ra 0$ and $\tH^1(\sco_\Sigma(i)) = 0$,
$\forall \, i \in \z$.

(e) Assume that $X = X_0 \cup X_1$, with $X_0$, $X_1$ locally Cohen-Macaulay curves such
that the scheme $X_0 \cap X_1$ is 0-dimensional, of length $r \geq 1$, and that it is
contained in a line $L \subset \piii$. If $\tH^1(\sci_{X_0}(i)) = 0$ and
$\tH^1(\sci_{X_1}(i)) = 0$ for $i \leq r-2$ then$\, :$
\[
\Delta^2\h^0_X = \Delta^2\h^0_{X_0} + \Delta^2\h^0_{X_1} -\chi_{\{0\}} + \chi_{\{r\}}\, , 
\]
where $\chi_A$ denotes, here, the \emph{characteristic function} of a subset $A \subset \z$.

\emph{Indeed}, consider the exact sequence
$0 \ra \sco_X \ra \sco_{X_0} \oplus \sco_{X_1} \ra \sco_{X_0 \cap X_1} \ra 0$. The maps
$\tH^0(\sco_\piii(i)) \ra \tH^0(\sco_{X_j}(i))$, $j = 0,\, 1$, are surjective for
$i \leq r-2$ and the map $\tH^0(\sco_\piii(i)) \ra \tH^0(\sco_{X_0 \cap X_1}(i))$ is
surjective for $i \geq r-1$. One deduces that the sequence$\, :$
\[
0 \ra \tH^0(\sco_X(i)) \ra \tH^0(\sco_{X_0}(i)) \oplus \tH^0(\sco_{X_1}(i)) \ra
\tH^0(\sco_{X_0 \cap X_1}(i)) \ra \tH^1(\sci_{X_0 \cap X_1}(i)) \ra 0 
\]
is exact, $\forall \, i \in \z$. Using the exact sequence
$0 \ra \sci_L \ra \sci_{X_0 \cap X_1} \ra \sco_L(-r) \ra 0$ one gets that
$\h^1(\sci_{X_0 \cap X_1}(i)) = \h^1(\sco_L(i-r))$ for $i \geq -1$. On the other hand, since
$X_0 \cap X_1$ is 0-dimensional of length $r$, $\h^1(\sci_{X_0 \cap X_1}(i)) = r$ 
for $i \leq -1$. 
\end{remark}

\begin{example}\label{Ex:xinq}
Let $0 \leq a \leq b$ be integers and let $X$ be an effective divisor of type $(a , b)$ on
a nonsingular quadric surface $\Sigma \subset \piii$, $\Sigma \simeq \pj \times \pj$.
Then $(\Delta^2\h_X^0)(i) = 0$ for $i < 0$ and for $i > a$ and
\[
(\Delta^2\h_X^0)(i) =
\begin{cases}
1,     &\text{if $i=0$;}\\
2,     &\text{if $0 < i < a$;}\\
b-a+1, &\text{if $i = a$.}
\end{cases}  
\]

\emph{Indeed}, one can assume, by Remark~\ref{R:Delta2h0}(d), that
$X = Y \cup L_1 \cup \ldots \cup L_{b-a}$, where $Y$ is the complete intersection of $\Sigma$
with a surface of degree $a$ and $L_1 , \ldots , L_{b-a}$ are mutually disjoint lines
belonging to the linear system $\vert \sco_\Sigma(0 , 1) \vert$. Then one has an exact
sequence$\, :$
\[
0 \lra {\textstyle \bigoplus}_{i=1}^{b-a}\sco_{L_i}(-a) \lra \sco_X \lra \sco_Y \lra 0 
\]
and the map $\tH^0(\sco_X(i)) \ra \tH^0(\sco_Y(i))$ is surjective, $\forall \, i \in \z$,
because $Y$ is arithmetically Cohen-Macaulay. 
\end{example}

\begin{example}\label{Ex:xincone}
Let $X$ be a curve o a quadric cone $\Sigma \subset \piii$. As it is well known, if $X$ has
even degree, let us say $2m$, then $X$ is the complete intersection of $\Sigma$ with a surface
of degree $m$ and if $X$ has odd degree, let us say $2m-1$, then it is directly linked to a
line contained in $\Sigma$ by the complete intersection of $\Sigma$ with a surface of degree
$m$. In both cases, $X$ is arithmetically Cohen-Macaulay. Using Remark~\ref{R:Delta2h0}(c) one
sees easily that in the former case $\Delta^2\h_X^0$ coincides with the corresponding function
of a divisor of type $(m , m)$ on a nonsingular quadric surface, while in the latter case the
same is true for a divisor of type $(m-1 , m)$. 
\end{example}  

\begin{remark}\label{R:xinh0h1}
Let $H_0,\, H_1 \subset \piii$ be two distinct planes and $L := H_0 \cap H_1$ their intersection
line. Assume that $H_i$ has equation $T_i = 0$, $i = 0,\, 1$ (recall that the projective
coordinate ring of $\piii$ is $S = k[T_0, \ldots , T_3]$). Consider a curve $X$ which is a
subscheme of $H_0 \cup H_1$.

If $L$ is not a component of $X$ then $X$ is the union of two plane curves intersecting in
finitely many points.

Assume, now, that $L$ is a component of $X$. Let $C_0$ be the curve defined by the ideal
sheaf $(\sci_X : T_1)$. Since $T_0T_1$ belongs to the homogeneous ideal $I(X)$ of $X$ it
follows that $C_0$ is a subscheme of $H_0$ hence $I(C_0) = (T_0\, ,\, T_1^{r_0}f_0)$, with
$r_0 \geq 0$ and $f_0 \in k[T_1 , T_2 , T_3]$ not divisible by $T_1$. Analogously, if $C_1$
is the curve defined by $(\sci_X : T_0)$ then $I(C_1) = (T_1\, ,\, T_0^{r_1}f_1)$, with
$r_1 \geq 0$ and $f_1 \in k[T_0 , T_2 , T_3]$ not divisible by $T_0$.
One has
$(I(X) : T_1)T_1 + (I(X) : T_0)T_0 \subseteq I(X)\subseteq (I(X) : T_1) \cap  (I(X) : T_0)$
hence
\[
(T_1^{r_0+1}f_0\, ,\, T_0T_1\, ,\, T_0^{r_1+1}f_1) \subseteq I(X) \subseteq
(T_1^{r_0}f_0\, ,\, T_0T_1\, ,\, T_0^{r_1}f_1)\, . 
\]
Let $X^\prime$ (resp., $X^{\prime \prime}$) be the curve defined by the homogeneous ideal situated
at the end (resp., beginning) of the above sequence of inclusions. They are both arithmetically
Cohen-Macaulay because, for example, one has a resolution$\, :$
\[
0 \lra 
\begin{matrix}
S(-r_0-2-d_0)\\
\oplus\\
S(-r_1-2-d_1)
\end{matrix}
\xra{{\left(\begin{smallmatrix} -T_0 & T_1^{r_0}f_0 & 0\\
0 & -T_0^{r_1}f_1 & T_1\end{smallmatrix}\right)}^{\text{t}}}
\begin{matrix} 
S(-r_0-1-d_0)\\
\oplus\\
S(-2)\\
\oplus\\ 
S(-r_1-1-d_1) 
\end{matrix}
\lra I(X^{\prime \prime}) \lra 0\, , 
\]
where $d_i$ is the degree of $f_i$, $i = 0,\, 1$. Moreover,
$\sci_L\sci_{X^\prime} \subseteq \sci_{X^{\prime \prime}}$ and
$\sci_{X^\prime}/\sci_{X^{\prime \prime}} \simeq \sco_L(-r_0-d_0) \oplus \sco_L(-r_1-d_1)$.

Since $X$ is locally Cohen-Macaulay one has
$\text{depth}(\sci_{X^\prime}/\sci_X)_x \geq 1$,
for any $x$ in the support of $\sci_{X^\prime}/\sci_X$, hence $\sci_{X^\prime}/\sci_X$ is a
locally free $\sco_L$-module.

\vskip2mm

\noindent
{\bf Claim.}\quad $X \neq X^\prime$.

\vskip2mm

\noindent
\emph{Proof of the claim.} Assume, by contradiction, that $X = X^\prime$. Since
$I(X) \subseteq I(L)$ it follows that $r_0 \geq 1$ and $r_1 \geq 1$. One has
$T_1^{r_0-1}f_0 \in (I(X^\prime) : T_1)$ and $T_1^{r_0-1}f_0 \notin (I(X) : T_1)$, and this is
a \emph{contradiction}. \hfill $\Diamond$

\vskip2mm

It follows, from the claim, that either $X = X^{\prime \prime}$ or
$\sci_{X^\prime}/\sci_X \simeq \sco_L(m)$, for some $m \in \z$. Since one has an epimorphism
$\sci_{X^\prime}/\sci_{X^{\prime \prime}} \ra \sci_{X^\prime}/\sci_X$ one deduces that either
$m = \text{min}(-r_0-d_0\, ,\, -r_1-d_1)$ or
$m \geq \text{max}(-r_0-d_0\, ,\, -r_1-d_1)$. 
\end{remark}

\begin{example}\label{Ex:xinh0h1}
Let $X$ be a curve contained (as a subscheme) in the union of two distinct planes $H_0$
and $H_1$ intersecting along a line $L$.

If $L$ is not a component of $X$ then $X = C_0 \cup C_1$, with $C_i \subset H_i$ of degree
$d_i$, $i = 0,\, 1$, such that the scheme $C_0\cap C_1$ is 0-dimensional, of length $r$.
Assuming that $d_0 \leq d_1$, hence that $r \leq d_0$, it follows from
Remark~\ref{R:Delta2h0}(e) that $(\Delta^2\h^0_X)(i) = 0$ for $i < 0$ and for $i \geq d_1$, 
while for $0 \leq i < d_1\, :$ 
\[
(\Delta^2\h^0_X)(i) =
\begin{cases}
1, &\text{if $i = 0$;}\\
2, &\text{if $1 \leq i < d_0$;}\\
1, &\text{if $d_0 \leq i < d_1$} 
\end{cases}
\  +\  
\begin{cases}
1, &\text{if $i = r$;}\\
0, &\text{if $i \neq r$.}  
\end{cases}  
\]

If $L$ is a component of $X$ then, with the notation from Remark~\ref{R:xinh0h1}, either
$X = X^{\prime \prime}$ or there is an exact sequence
$0 \ra \sco_L(m) \ra \sco_X \ra \sco_{X^\prime} \ra 0$. In the former case, assuming that
$r_0 + d_0 \leq r_1 + d_1$, one deduces, from Remark~\ref{R:Delta2h0}(c), that
$(\Delta^2\h^0_{X^{\prime \prime}})(i) = 0$ for $i < 0$ and for $i > r_1 + d_1$, while for
$0 \leq i \leq r_1 + d_1\, :$
\[
(\Delta^2\h^0_{X^{\prime \prime}})(i) =
\begin{cases}
1, &\text{if $i = 0$;}\\
2, &\text{if $1 \leq i \leq r_0 + d_0$;}\\
1, &\text{if $r_0 + d_0 < i \leq r_1 + d_1$.} 
\end{cases}  
\]
In the latter case, the map $\tH^0(\sco_X(i)) \ra \tH^0(\sco_{X^\prime}(i))$ is surjective,
$\forall \, i \in \z$, because $X^\prime$ is arithmetically Cohen-Macaulay. The formula for 
$\Delta^2\h^0_{X^\prime}$ can be obtained from the above formula for
$\Delta^2\h^0_{X^{\prime \prime}}$ replacing $r_i$ by $r_i - 1$, $i = 0,\, 1$, while
\[
(\Delta^2\h^0_L)(i+m) =
\begin{cases}
1, &\text{if $i = -m$;}\\
0, &\text{if $i \neq -m$.}   
\end{cases}  
\]
Recall, from Remark~\ref{R:xinh0h1}, that $-m = r_1 + d_1$ or $-m \leq r_0 + d_0$. Notice, also,
that $\tH^1(\sci_X) = 0$ if and only if $\tH^0(\sco_L(m)) = 0$ if and only if $-m > 0$. 
\end{example}

In order to describe $\Delta^2\h^0_X$ for curves $X$ contained in a double plane one has to
be able to handle the Hilbert functions of 0-dimensional subschemes of the projective
plane. It seems to us that the most convenient way of doing that is by using standard
resolutions. We recall, in the next remark, the necessary general facts about initial
ideals, Galligo's theorem, the division algorithm and Buchberger's criterion. For details
one can consult Green \cite{gr} and/or Chapter 15 in Eisenbud's textbook \cite{eis}.

\begin{remark}\label{R:gin}
Order the monomials in $S = k[T_0, \ldots , T_n]$ according to the
\emph{reverse lexicographic order} for which $T_0 > \cdots > T_n$ and
$T_0^{a_0}\ldots T_n^{a_n} > T_0^{b_0} \ldots T_n^{b_n}$ if there exists an index $i$ such that
$a_i < b_i$ and $a_j = b_j$ for $j > i$ (assuming that the two monomials have the same
degree). If $f \in S \setminus \{0\}$ is a homogeneous polynomial one denotes by
$\text{in}(f)$ the largest monomial appearing in the expression of $f$ as a sum of
non-zero monomials. If $I \subset S$ is a homogeneous ideal, one denotes by $\text{in}(I)$
the ideal generated by $\text{in}(f)$ for $f \in I \setminus \{0\}$. $I$ and $\text{in}(I)$
have the same Hilbert function.

$\text{GL}(n+1)$ acts to the left on $S$ by automorphisms of $k$-algebras by the matrix
formula$\, :$
\[
(\gamma \cdot T_0 , \dots , \gamma \cdot T_n) := (T_0 , \ldots , T_n)\gamma\, ,\
\gamma \in \text{GL}(n+1)\, . 
\]
If $f$ is an element of $S$ and if ${\widetilde f} \colon k^{n+1} \ra k$ is the polynomial
function defined by $f$ then
$(\gamma \cdot f)\sptilde = {\widetilde f} \circ \gamma^{\text{t}}$ (where $\gamma^{\text{t}}$ is
the automorphism of $k^{n+1}$ defined by the transpose of the matrix defining $\gamma$).

If $I \subset S$ is a homogeneous ideal there exists a non-empty open subset
$\mathcal{U}$ of $\text{GL}(n+1)$ such that
$\text{in}(\gamma \cdot I) = \text{in}(\gamma^\prim \cdot I)$,
$\forall \, \gamma,\, \gamma^\prim \in \mathcal{U}$. The \emph{generic initial ideal} of
$I$ is defined by $\text{gin}(I) := \text{in}(\gamma \cdot I)$, $\gamma \in \mathcal{U}$.
One says that $I$ is in \emph{generic coordinates} if the identity map of $k^{n+1}$ belongs
to $\mathcal{U}$.

The theorem of Galligo asserts that $\beta \cdot \text{gin}(I) = \text{gin}(I)$, for every
$\beta$ in the Borel subgroup $B$ of $\text{GL}(n+1)$ consisting of upper triangular
matrices. In characteristic $0$ this is equivalent to $\text{gin}(I)$ being
\emph{strongly stable}, which means that if $T_0^{a_0} \ldots T_n^{a_n} \in \text{gin}(I)$
and $a_i > 0$ then $(T_j/T_i) \cdot T_0^{a_0} \ldots T_n^{a_n} \in \text{gin}(I)$,
$\forall \, j < i$. For the slightly more general notion of \emph{stable monomial ideal}
see Eliahou and Kervaire \cite{ek}. 

Assume, from now on, that $I$ is in generic coordinates hence, in particular, that
$\text{in}(I)$ is strongly stable. Then
\[
I \text{ is saturated } \iff \text{in}(I) \text{ is saturated } \iff
T_n \text{ is $S/\text{in}(I)$-regular} 
\]
and in this case $\text{in}(I)$ is generated by monomials in $k[T_0 \ldots , T_{n-1}]$.
Let $\Delta$ be a finite set of monomials generating $\text{in}(I)$ (as an ideal) and
choose, for each $u \in \Delta$, a homogeneous polynomial $f_u \in I$ such that
$\text{in}(f_u) = u$ and none of the monomials appearing in $f_u - u$ belongs to
$\text{in}(I)$. $(f_u)_{u \in \Delta}$ is called a
\emph{reduced Gr\"{o}bner} (\emph{or standard}) \emph{basis} of $I$. It generates $I$ as an ideal
(but not necessarily in a minimal way).

The \emph{division algorithm} says that if $f \in S$ is a homogeneous polynomial then there are
uniquely determined homogeneous polynomials $q_u$, $u \in \Delta$, and $r$, of appropriate
degree, such that$\, :$
\begin{enumerate}
\item[(i)] $f = {\textstyle \sum}_{u \in \Delta}q_uf_u + r\, ;$
\item[(ii)] $q_u \in k[T_{\text{max}(u)} , \ldots , T_n]$, $\forall \, u \in \Delta\, ;$
\item[(iii)] None of the monomials appearing in $r$ belongs to $\text{in}(I)$. 
\end{enumerate}
As a matter of notation, if $u = T_0^{a_0} \ldots T_n^{a_n}$ then
$\text{max}(u) := \text{max}\{i \vb a_i > 0\}$.

Now, the module of relations between the elements of $\Delta$ has a system of generators
consisting of relations of the form$\, :$
\[
\frac{w}{u} \cdot u - \frac{w}{v} \cdot v = 0
\]
with $u,\, v \in \Delta$ and with $w$ the least common multiple of $u$ and $v$. According to
\emph{Buchberger's criterion}, the division algorithm produces a relation$\, $
\[
\frac{w}{u} \cdot f_u - \frac{w}{v} \cdot f_v =
{\textstyle \sum}_{\mu \in \Delta}q_\mu f_\mu 
\]
between the generators $f_\mu$, $\mu \in \Delta$, of $I$. According to a result of
Schreyer \cite{sch}, 
these relations generate the module of relations between $f_\mu$, $\mu \in \Delta$. 
\end{remark}

\begin{remark}\label{R:igamma}
We particularize the general facts recalled in the preceding remark to the case $n = 2$
and $I = I(\Gamma)$, where $\Gamma$ is a 0-dimensional subscheme of $\pii$. Assume that $I$
is in generic coordinates. Then $\text{in}(I)$ has a system of generators of the form$\, :$
\[
T_0^\sigma\, ,\, T_0^{\sigma - 1}T_1^{\lambda_1}\, ,\, \ldots \, ,\, T_1^{\lambda_\sigma}\, ,
\text{ with } 0 = \lambda_0 < \lambda_1 < \cdots < \lambda_\sigma \, , 
\]
where $\sigma := \text{min}\{l \geq 1 \vb \tH^0(\sci_\Gamma(l)) \neq 0\}$. Consider a
reduced Gr\"{o}bner basis $(f_i)_{0 \leq i \leq \sigma}$ of $I$, such that
$\text{in}(f_i) = T_0^{\sigma - i}T_1^{\lambda_i}$. Then$\, :$
\[
T_0 \cdot T_0^{\sigma - i}T_1^{\lambda_i} -
T_1^{\lambda_i - \lambda_{i-1}} \cdot T_0^{\sigma - i + 1}T_1^{\lambda_{i-1}} = 0\, ,\
i = 1, \ldots , \sigma\, , 
\]
is a minimal system of generators of the module of relations between the above generators
of $\text{in}(I)$. Using the result of Schreyer recalled at the end of Remark~\ref{R:gin},
one gets a graded free resolution$\, :$
\[
0 \lra {\textstyle \bigoplus}_{j=1}^\sigma S(-d_j-1) \lra
{\textstyle \bigoplus}_{i=0}^\sigma S(-d_i) \lra I \lra 0\, , 
\]
where $d_i := \sigma - i + \lambda_i$. We call it the \emph{standard resolution} of $I$.
Sheafifying, one gets a resolution$\, :$
\[
0 \lra {\textstyle \bigoplus}_{j=1}^\sigma \sco_\pii(-d_j-1) \lra
{\textstyle \bigoplus}_{i=0}^\sigma \sco_\pii(-d_i) \lra \sci_\Gamma \lra 0\, .
\]
In particular,
$\text{deg}\, \Gamma = \h^1(\sci_\Gamma(-1)) = d_1 + \cdots + d_\sigma - \sigma(\sigma -1)/2$.
\end{remark}

\begin{example}\label{Ex:xin2h}
Let $H \subset \piii$ be a plane and let $X$ be a curve contained, as a subscheme, in the
effective divisor $2H$ on $\piii$. According to Hartshorne and Schlesinger \cite{has}, there
exist plane curves $C_0 \subseteq C \subset H$ (that is, effective divisors on $H$), of
degree $r_0$ and $r$, respectively, and a 0-dimensional subscheme $\Gamma$ of $C_0$, locally
complete intersection in $H$, such that $\sci_X$ is the kernel of an epimorphism
$\sci_C \ra \sch om_{\sco_{C_0}}(\sci_{\Gamma , C_0} , \sco_{C_0}(-1))$. Consider the standard
resolution of $\sci_{\Gamma , H}$ from Remark~\ref{R:igamma} ($H \simeq \pii$). Since $\Gamma$
is contained in $C_0$ one has $r_0 \geq \sigma$. One deduces a resolution$\, :$
\[
0 \lra \sco_H(-r_0) \oplus {\textstyle \bigoplus}_{j=1}^\sigma \sco_H(-d_j-1) \lra
{\textstyle \bigoplus}_{i=0}^\sigma \sco_H(-d_i) \lra \sci_{\Gamma , C_0} \lra 0\, .
\]
Using the isomorphism of ``d\'{e}calage''
$\sch om_{\sco_{C_0}}(\sci_{\Gamma , C_0} , \omega_{C_0}) \simeq
\sce xt_{\sco_H}^1(\sci_{\Gamma , C_0} , \omega_H)$,
one gets that$\, :$ 
\[
\sch om_{\sco_{C_0}}(\sci_{\Gamma , C_0} , \sco_{C_0})(r_0) \simeq
\sce xt_{\sco_H}^1(\sci_{\Gamma , C_0} , \sco_H)\, . 
\]
Dualizing the above resolution of $\sci_{\Gamma , C_0}$ one, consequently, obtains a
resolution$\, :$
\begin{gather*}
0 \lra {\textstyle \bigoplus}_{i=0}^\sigma \sco_H(d_i - r_0 - 1) \lra
\sco_H(-1) \oplus {\textstyle \bigoplus}_{j=1}^\sigma \sco_H(d_j - r_0) \lra\\
\lra \sch om_{\sco_{C_0}}(\sci_{\Gamma , C_0} , \sco_{C_0}(-1)) \lra 0\, . 
\end{gather*}
Consider, now, the exact sequence$\, :$
\[
0 \lra \sch om_{\sco_{C_0}}(\sci_{\Gamma , C_0} , \sco_{C_0}(-1)) \lra \sco_X \lra
\sco_C \lra 0\, . 
\]
One deduces, firstly, that
$\tH^1(\sci_X) = 0 \iff \tH^0(\sch om_{\sco_{C_0}}(\sci_{\Gamma , C_0} , \sco_{C_0}(-1))) = 0
\iff r_0 > d_\sigma$.
Assume, from now on, that $r_0 > d_\sigma$. The map $\tH^0(\sco_X(i)) \ra \tH^0(\sco_C(i))$ is
surjective, $\forall \, i \in \z$, because $C$ is arithmetically Cohen-Macaulay. Using the
above resolution of $\sch om_{\sco_{C_0}}(\sci_{\Gamma , C_0} , \sco_{C_0}(-1))$, one gets that
$\h^0(\sco_X(i))$ equals$\, :$ 
\[
\h^0(\sco_C(i)) + [\h^0(\sco_H(i-1)) - \h^0(\sco_H(i + \sigma - r_0 -1))] +
{\textstyle \sum}_{j=1}^\sigma \h^0(\sco_L(i + d_j - r_0)), 
\]
where $L \subset H$ is a line.
If $F(i) : = \h^0(\sco_H(i-1)) - \h^0(\sco_H(i + \sigma - r_0 -1))$ and $f := \Delta F$ then
$f(i) = \h^0(\sco_L(i-1)) - \h^0(\sco_L(i + \sigma - r_0 -1))$. Using the last formula in
Remark~\ref{R:Delta2h0}(c), one gets that $(\Delta^2\h^0_X)(i) = 0$ for $i < 0$ and for
$i \geq r$ and that$\, :$
\[
(\Delta^2\h^0_X)(i) =
\begin{cases}
1, &\text{if $i = 0$;}\\
2 + \text{card}\, \{1 \leq j \leq \sigma \vb d_j = r_0 - i\},
&\text{if $1 \leq i \leq r_0 - \sigma$;}\\
1, &\text{if $r_0 - \sigma < i < r$.} 
\end{cases}  
\]
\end{example}

\begin{proposition}\label{P:cards(i)=1}
Let $E$ be a stable rank $2$ vector bundle on $\piii$ with $c_1 = 0$. Assume, with the
notation from Definition~\emph{\ref{D:spectrum}}, that $s(0) = 1$, $s(1) = 2$ and
$s(2) = 2$. Then, except for the case where there exists $m \geq 3$ such that
$s(i) = 2$ for $1 \leq i < m$, $s(m) \geq 4$ and $s(i) = 0$ for $i > m$, one has
\[
{\fam0 card}\, \{i \geq 1 \vb s(i) = 1\} \geq
{\textstyle \sum}_{i \geq 1}{\fam0 max}(s(i) - 2 , 0) - 1\, . 
\]
\end{proposition}

\begin{proof}
It follows, from Proposition~\ref{P:s(0)=1}, that $E(1)$ has a global section whose zero
scheme $Y$ is a double structure on a curve $X$ with $\tH^1(\sci_X) = 0$ such that $\sci_Y$
is the kernel of an epimorphism $\sci_X \ra \omega_X(2)$. Moreover, by Corollary~\ref{C:s(0)=1},
one has $\tH^0(\sci_X(2)) \neq 0$. One deduces, from Lemma~\ref{L:Delta2h0}, that
$s(i) = (\Delta^2\h^0_X)(i)$, $\forall \, i \geq 0$.

If $X$ is contained in a nonsingular quadric surface then Example~\ref{Ex:xinq} shows that
either the spectrum of $E$ is $(\ldots , 0 , 1^2 , 2^2)$ or there exists an integer
$m \geq 3$ such that $s(0) = 1$, $s(i) = 2$ for $1 \leq i < m$, $s(m) \geq 1$ and $s(i) = 0$
for $i > m$. If $X$ is contained in a cone the same holds except that
$s(m) \in \{1 , 2\}$ (see Example~\ref{Ex:xincone}). 

If $X$ is contained in the union of two planes then Example~\ref{Ex:xinh0h1} implies that
$s(i) \leq 3$, $\forall \, i \geq 0$, and $s(i) = 3$ for at most one $i \geq 0$.

Assume, finally, that $X$ is contained in a double plane. Using the notation from
Example~\ref{Ex:xin2h}, the condition $s(1) = s(2) = 2$ implies that $r_0 - 2 > d_\sigma$, 
by the last formula in that example. By the same formula$\, :$
\[
{\textstyle \sum}_{i \geq 1}\text{max}(s(i) - 2 , 0) = \sigma \text{ and }
\text{card}\, \{i \geq 1 \vb s(i) = 1\} = (r-1) - (r_0 - \sigma)\, .   
\]
Since $r \geq r_0$ one gets the relation from the conclusion of the proposition. 
\end{proof}

\begin{example}\label{Ex:q2har}
It follows, from Proposition~\ref{P:cards(i)=1}, that there is no stable rank 2 vector
bundle on $\piii$ with $c_1 = 0$, $c_2 = 21$, and spectrum
$(\ldots , 0 , 1^2 , 2^2, 3^4 , 4^2)$ (compare with Hartshorne and Rao
\cite[Prop.~2.15]{har}). The same is true for the spectrum
$(\ldots , 0 , 1^2, 2^2 , 3^3 , 4^3)$. 
\end{example}

\section{Spectra and double structures on space curves, II}\label{S:doubleii}

We examine, in this section, the case of the stable rank 2 vector bundles $E$ on
$\piii$ with $c_1 = 0$ whose spectrum satisfies the property that $s(0) = 2$,
$s(1) = 2$ and $s(2) \geq 2$. We shall need the following lemma, whose proof is
postponed to Appendix~\ref{A:h0e(1)=2}.

\begin{lemma}\label{L:h0e(1)=2}
Let $E$ be a stable rank $2$ vector bundle on $\piii$ with $c_1 = 0$ and $c_2 \geq 2$.
Then ${\fam0 h}^0(E(1)) \leq 2$ and if ${\fam0 h}^0(E(1)) = 2$ then the zero scheme of an
arbitrary non-zero global section of $E(1)$ is a curve $X$ such that one of the following
holds$\, :$

\emph{(i)} $X$ is a divisor of type $(d , 0)$ or $(0 , d)$ one a nonsingular quadric
surface $Q \simeq \pj \times \pj\, ;$

\emph{(ii)} $X$ is a double structure on a plane curve $C$ of degree $d$ whose ideal
sheaf $\sci_X$ is the kernel of an epimorphism $\sci_C \ra \omega_C(2)$. Recall that
plane curve means effective divisor on some plane in $\piii$. 
\end{lemma}

The inequality $\h^0(E(1)) \leq 2$ in the statement of the lemma is well known$\, :$
see Ellia and Gruson \cite{eg}. The characterization of the bundles with $\h^0(E(1)) = 2$
might be, also, well known but we are not aware of a reference. Notice that in
case (i) the spectrum of $E$ is $(0^{d-1})$ while in case (ii) it is
$(\ldots , 0, 1, 2, \ldots , d-1)$ (by Lemma~\ref{L:Delta2h0}).

\begin{proposition}\label{P:s(0)=s(1)=2}
Let $E$ be a stable rank $2$ vector bundle on $\piii$ with $c_1 = 0$. Consider the
spectrum of $E$ \emph{(}see Definition~\emph{\ref{D:spectrum}}\emph{)} and its minimal
\emph{(}resp., reduced\emph{)} semi-splitting monad
\emph{(}see Definition~\emph{\ref{D:rho}}\emph{)}.
Assume that $s(0) = 2$, $s(1) = 2$ and $s(2) \geq 2$ hence that, by
Corollary~\emph{\ref{C:si=2}(b)}, $b_0 = 2$, $\rho(-1) \leq 1$ and $\rho(0) = 0$. 
Let us denote the morphism
$({\overline \beta}_+ , \beta_0) \colon {\overline B}_+ \oplus B_0 \ra A_-^\vee$ by
${\overline \beta}_{\geq 0}$. 

\vskip2mm 

\emph{(a)} If $\rho(-1) = 0$ then there exists a trivial rank $1$ subbundle $B_0^\prime$
of $B_0$ such that the degeneracy scheme of the restriction 
${\overline \beta}_{\geq 0}^\prim \colon {\overline B}_+ \oplus B_0^\prime \ra A_-^\vee$ of
${\overline \beta}_{\geq 0}$ 
is a locally Cohen-Macaulay curve $X$ with ${\fam0 h}^0(\sci_X(2)) = 1$ and
${\fam0 H}^1(\sci_X(l)) = 0$ for $l \leq 1$. In this case, by
Lemma~\emph{\ref{L:splittingdouble}}, $E(2)$ has a global section whose zero scheme is a 
double structure $Y$ on $X$ such that the ideal sheaf $\sci_Y$ is the kernel of an
epimorphism $\sci_X \ra \omega_X$.

\vskip2mm 

\emph{(b)} If $\rho(-1) = 1$ one has $A_-^\vee = A_{\leq -2}^\vee \oplus \sco_\piii(1)$ and
the image of ${\overline \beta}_+ \colon {\overline B}_+ \ra A_-^\vee$ is contained in
$A_{\leq -2}^\vee$. Let
${\overline \beta}_{2+} \colon {\overline B}_+ \ra A_{\leq -2}^\vee$ be the morphism coinduced
by ${\overline \beta}_+$. 
Then the degeneracy scheme of ${\overline \beta}_{2+}$ 
$($resp., 
${\overline \beta}_{\geq 0})$ is a
locally Cohen-Macaulay curve $X^\prime$ $($resp., $X)$ such that
${\fam0 H}^1(\sci_{X^\prime}(l)) = 0$ for $l \leq 1$, $X^\prime \subset X$ as schemes and
$\sci_{X^\prime}/\sci_X \simeq \sco_L$, for some line $L \subset \piii$. Moreover, $E(1)$
has a global section with zero scheme $Y$ such that the ideal sheaf $\sci_Y$ is the
kernel of an epimorphism $\sci_{X^\prime} \ra \omega_X(2)$. 
\end{proposition}

\begin{proof}
Let $B_0^\prime$ be an arbitrary trivial rank 1 subbundle of $B_0$.
As in the proof of Proposition~\ref{P:s(0)=1} (with ${\overline \beta}_+$ replaced by
${\overline \beta}_{\geq 0}^\prim$ and ${\overline B}_+$ replaced by
${\overline B}_+ \oplus B_0^\prime$ and taking into account that $s(0) = 2$) one deduces that
$\Cok {\overline \beta}_{\geq 0}^\prim$ is a torsion sheaf and that
$c_1(\Cok {\overline \beta}_{\geq 0}^\prim) \leq 1$.

\vskip2mm

\noindent
{\bf Claim 1.}\quad \emph{If} $\rho(-1) = 1$ \emph{then}
$c_1(\Cok {\overline \beta}_{\geq 0}^\prim) = 1$, $\forall \, B_0^\prime \subset B_0$
\emph{as above}. \emph{Moreover, the degeneracy scheme of} ${\overline \beta}_{2+}$
\emph{has the expected codimension} 2 \emph{in} $\piii$,
$\Ker {\overline \beta}_{2+} \simeq \sco_\piii(-1)$ \emph{and the component}
$\beta_{10} \colon B_0 \ra \sco_\piii(1)$ \emph{of} $\beta_0$ \emph{is defined by two
linearly independent linear forms}.

\vskip2mm

\noindent
\emph{Proof of Claim} 1. Since
${\overline \beta}_{\geq 0}^\prim \colon {\overline B}_+ \oplus B_0^\prime \ra A_-^\vee$ maps
${\overline B}_+$ into $A_{\leq -2}^\vee$, one gets an exact sequence$\, :$
\[
\Ker \beta_{10}^\prim \lra \Cok {\overline \beta}_{2+}
\lra \Cok {\overline \beta}_{\geq 0}^\prim \lra
\Cok \beta_{10}^\prim \lra 0\, ,  
\]
where $\beta_{10}^\prim \colon B_0^\prime \ra \sco_\piii(1)$ denotes the restriction of
$\beta_{10}$. Since $\Cok {\overline \beta}_{\geq 0}^\prim$ is a torsion sheaf, the morphism
$\beta_{10}^\prim$ must be non-zero, hence $\Ker \beta_{10}^\prim = 0$ and
$\Cok \beta_{10}^\prim \simeq \sco_H(1)$, for some plane $H \subset \piii$.

If $\Cok {\overline \beta}_{2+}$ would not be a torsion sheaf, Remark~\ref{R:c1scf}(b)
would imply that $c_1(\Cok {\overline \beta}_{2+}) \geq 2$ and this would contradict the
fact that $c_1(\Cok {\overline \beta}_{\geq 0}^\prim) \leq 1$. Consequently,
$\Cok {\overline \beta}_{2+}$ is a torsion sheaf. In this case
$c_1(\Cok {\overline \beta}_{2+}) \geq 0$. One deduces that
$c_1(\Cok {\overline \beta}_{\geq 0}^\prim) = 1$ and $c_1(\Cok {\overline \beta}_{2+}) = 0$,
hence the degeneracy locus of ${\overline \beta}_{2+}$ has (the expected) codimension 2 in
$\piii$. Since $\text{rk}\, {\overline B}_+ = \text{rk}\, A_-^\vee = \text{rk}\, A_{\leq -2}^\vee +1$
and 
$c_1(A_{\leq -2}^\vee) - c_1({\overline B}_+) = c_1(A_-^\vee) - 1 - c_1({\overline B}_+) = 1$ one also
gets that $\Ker {\overline \beta}_{2+} \simeq \sco_\piii(-1)$. 

Finally, since the restriction of $\beta_{10}$ to every rank 1 trivial subbundle $B_0^\prime$
of $B_0$ is non-zero, $\beta_{10}$ is defined by two linearly independent linear forms.
\hfill $\Diamond$

\vskip2mm

\noindent
{\bf Claim 2.}\quad \emph{If} $\rho(-1) = 0$ \emph{then} $\Cok {\overline \beta}_+$
\emph{is a torsion sheaf}.

\vskip2mm

\noindent
\emph{Proof of Claim} 2. ${\overline B}_+$ and $A_-^\vee$ have the same rank and
$c_1(A_-^\vee) - c_1({\overline B}_+) = s(0) = 2$. Assume, by contradiction, that
$\Cok {\overline \beta}_+$ is not a torsion sheaf. Choose a rank 1 trivial subbundle
$B_0^\prime$ of $B_0$. One has an exact sequence$\, :$
\[
B_0^\prime \lra \Cok {\overline \beta}_+ \lra \Cok {\overline \beta}_{\geq 0}^\prim
\lra 0\, . 
\]
Since, according to our assumption, $\Cok {\overline \beta}_+$ is a locally free sheaf
of rank $\geq 1$ on a non-empty open subset of $\piii$ and since
$\Cok {\overline \beta}_{\geq 0}^\prim$ is a torsion sheaf it follows that
$B_0^\prime \ra \Cok {\overline \beta}_+$ is an isomorphism on a non-empty open subset
of $\piii$ hence it is a monomorphism on $\piii$. Since $\rho(-1) = 0$,
Remark~\ref{R:c1scf}(b) implies that $c_1(\Cok {\overline \beta}_+) \geq 2$ and this
\emph{contradicts} the fact that $c_1(\Cok {\overline \beta}_{\geq 0}^\prim) \leq 1$.
\hfill $\Diamond$

\vskip2mm

(a) It suffices, actually, to show that there exists a trivial rank 1 subbundle $B_0^\prime$
of $B_0$ such that $c_1(\Cok {\overline \beta}_{\geq 0}^\prim) = 0$. Assume, by contradiction,
that $c_1(\Cok {\overline \beta}_{\geq 0}^\prim) = 1$, $\forall \, B_0^\prime \subset B_0$ rank 1
trivial subbundle. By Claim 2, $\Cok {\overline \beta}_+$ is a torsion sheaf. Since
${\overline B}_+$ and $A_-^\vee$ have the same rank it follows that
${\overline \beta}_+ \colon {\overline B}_+ \ra A_-^\vee$ is a monomorphism.

Let $B_0^\prime$ be an arbitrary trivial rank 1 subbundle of $B_0$. Since
$c_1(\Cok {\overline \beta}_{\geq 0}^\prim) = 1$ and since
$c_1(A_-^\vee) - c_1({\overline B}_+ \oplus B_0^\prime) = s(0) = 2$, it follows that the kernel
of ${\overline \beta}_{\geq 0}^\prim \colon {\overline B}_+ \oplus B_0^\prime \ra A_-^\vee$ is
isomorphic to $\sco_\piii(-1)$. One deduces an exact sequence$\, :$
\[
0 \lra \sco_\piii(-1) \lra B_0^\prime \lra \Cok {\overline \beta}_+ \lra
\Cok {\overline \beta}_{\geq 0}^\prim \lra 0\, ,
\]
from which one derives an exact sequence
$0 \ra \sco_H \ra \Cok {\overline \beta}_+ \ra \Cok {\overline \beta}_{\geq 0}^\prim \ra 0$,
for some plane $H \subset \piii$, of equation $h = 0$. It follows that the image of
$B_0^\prime \ra \Cok {\overline \beta}_+$ is annihilated by $h$. Moreover, one deduces that
$c_1(\Cok {\overline \beta}_+) = 2$ and that $H$ is an irreducible component of the support
of $\Cok {\overline \beta}_+$. It follows that there is a plane $H^\prim \subset \piii$, of
equation $h^\prime = 0$, such that $\text{Supp}\, \Cok {\overline \beta}_+ = H \cup H^\prim$
($H^\prim$ is not necessarily different from $H$).

Consequently, for any $B_0^\prime$ rank 1 trivial subbundle of $B_0$, the image of
$B_0^\prime \ra \Cok {\overline \beta}_+$ is annihilated either by $h$ or by $h^\prime$.
There exist distinct rank 1 trivial subbundles $B_0^\prime$ , $B_0^{\prime \prime}$ of $B_0$
such that the images of both of the morphisms $B_0^\prime \ra \Cok {\overline \beta}_+$,
$B_0^{\prime \prime} \ra \Cok {\overline \beta}_+$ are annihilated by the same linear form.
Let us say that this linear form if $h$. It follows that the image of the composite
morphism $B_0 \overset{\beta_0}{\lra} A_-^\vee \ra \Cok {\overline \beta}_+$ is annihilated
by $h$.
Since ${\overline \beta}_+ \colon {\overline B}_+ \ra A_-^\vee$ is a monomorphism, one has 
an exact sequence$\, :$
\[
0 \lra \Ker {\overline \beta}_{\geq 0} \lra B_0 \lra \Cok {\overline \beta}_+ \lra
\Cok {\overline \beta}_{\geq 0} \lra 0\, , 
\]
hence $hB_0(-1) \subseteq \Ker {\overline \beta}_{\geq 0}$. But
$\tH^0((\Ker {\overline \beta}_{\geq 0})(1)) \subseteq \tH^0(E(1))$ because $\tH^0(A_-(1)) = 0$
(due to the hypothesis $\rho(-1) = 0$), hence $\h^0(E(1)) \geq 2$. Using
Lemma~\ref{L:h0e(1)=2}, one gets a \emph{contradiction}.

\vskip2mm

(b) By Claim 1, the degeneracy scheme $X^\prime$ of
${\overline \beta}_{2+} \colon {\overline B}_+ \ra A_{\leq -2}^\vee$ has codimension 2 in $\piii$, 
$\Ker {\overline \beta}_{2+} \simeq \sco_\piii(-1)$ (hence
$\Cok {\overline \beta}_{2+} \simeq \omega_{X^\prime}(3)$) 
and the component $\beta_{10} \colon B_0 \ra \sco_\piii(1)$ of $\beta_0 \colon B_0 \ra A_-^\vee$
is defined by two linearly independent linear forms $h$ and $h^\prime$. Let $L \subset \piii$
be the line of equations $h = h^\prime = 0$. Recall that
$A_- = \sco_\piii(-1) \oplus A_{\leq -2}$ hence $A_-^\vee = A_{\leq -2}^\vee \oplus \sco_\piii(1)$. 
Let us denote the reduced semi-splitting monad of $E$ by $M^\bullet$. It has the following
two subcomplexes$\, :$
\[
K^{\prime \bullet}\  :\  0 \lra {\overline B}_+ \xra{{\overline \beta}_{2+}}
A_{\leq -2}^\vee \, ,\  
K^\bullet \  :\  \sco_\piii(-1) \xra{({\overline \alpha}_{+1}\, ,\, \alpha_{01})^{\text{t}}}
{\overline B}_+ \oplus B_0 \xra{{\overline \beta}_{\geq 0}} A_-^\vee\, . 
\]
The corresponding quotient complexes are$\, :$
\[
M^\bullet/K^{\prime \bullet} \  :\  A_- \xra{{\overline \alpha}_{\leq 0}}
B_0 \oplus {\overline B}_+^\vee
\xra{(\beta_{10}\, ,\, {\overline \beta}_{1-})} \sco_\piii(1)\, ,\  
M^\bullet/K^\bullet \  :\  A_{\leq -2} \xra{{\overline \alpha}_{-2}}
{\overline B}_+^\vee \lra 0\, . 
\]
The isomorphism $M^\bullet \izo M^{\bullet \vee}$ at the end of
Definition~\ref{D:semisplitmonad} induces isomorphisms
$M^\bullet/K^{\prime \bullet} \izo K^{\bullet \vee}$ and
$M^\bullet/K^\bullet \izo K^{\prime \bullet \vee}$. Moreover, $K^\bullet/K^{\prime \bullet}$ is
the complex
\[
\sco_\piii(-1) \overset{\alpha_{01}}{\lra} B_0 \overset{\beta_{10}}{\lra} \sco_\piii(1)\, .
\]

Using the short exact sequence of complexes
$0 \ra K^{\prime \bullet} \ra K^\bullet \ra K^\bullet/K^{\prime \bullet} \ra 0$ one gets an
isomorphism
$\sco_\piii(-1) \simeq \Ker {\overline \beta}_{2+} \izo \mathcal{H}^0(K^\bullet)$ and an
exact sequence$\, :$ 
\[
0 \lra \Cok {\overline \beta}_{2+} \lra \Cok {\overline \beta}_{\geq 0} \lra \sco_L(1)
\lra 0\, , 
\]
from which one deduces that the degeneracy scheme $X$ of ${\overline \beta}_{\geq 0}$ has
codimension 2 in $\piii$. The morphism
$({\overline \alpha}_{+1}\, ,\, \alpha_{01})^{\text{t}}\colon\sco_\piii(-1)\ra{\overline B}_+\oplus B_0$
is a locally split monomorphism hence its cokernel $C$ is a locally free sheaf on $\piii$.
${\overline \beta}_{\geq 0}$ factorizes as
${\overline B}_+ \oplus B_0 \twoheadrightarrow C \overset{\gamma}{\lra} A_-^\vee$. Since
$\text{rk}\, C = \text{rk}\, A_-^\vee + 1$, it follows that, actually, $X$ is a locally
Cohen-Macaulay curve. Since, as we noticed above,
$\Ker \gamma = \mathcal{H}^0(K^\bullet) \simeq \sco_\piii(-1)$, one deduces that
$\Cok {\overline \beta}_{\geq 0} = \Cok \gamma \simeq \omega_X(3)$. 

Using the exact sequence of complexes
$0 \ra K^\bullet/K^{\prime \bullet} \ra M^\bullet/K^{\prime \bullet} \ra M^\bullet/K^\bullet \ra 0$
and the fact that $M^\bullet/K^{\prime \bullet} \izo K^{\bullet \vee}$ and
$M^\bullet/K^\bullet \izo K^{\prime \bullet \vee}$ one gets an exact sequence$\, :$
\[
0 \lra \sci_X(1) \lra \sci_{X^\prime}(1) \lra \sco_L(1) \lra 0\, . 
\]

Finally, using the exact sequence
$0 \ra K^\bullet \ra M^\bullet \ra M^\bullet/K^\bullet \ra 0$ and the fact that
$M^\bullet/K^\bullet \izo K^{\prime \bullet \vee}$ one gets an exact sequence$\, :$ 
\[
0 \lra \sco_\piii(-1) \lra E \lra \sci_{X^\prime}(1) \lra \omega_X(3) \lra 0\, . 
\]
The fact that $\tH^1(\sci_{X^\prime}(l)) = 0$ for $l \leq 1$ follows using the exact sequences$\, :$
\[
0 \ra A_{\leq -2} \ra {\overline B}_+^\vee \ra \sci_{X^\prime}(1) \ra 0\, ,\
0 \ra {\overline B}_+^\vee \ra B_+^\vee \ra A_{\geq 0}^\vee \ra 0\, , 
\]
and the fact that $\tH^0(A_{\geq 0}^\vee) = 0$ because $\rho(0) = 0$. 
\end{proof}

\begin{corollary}\label{C:Delta2h0xprim}
Under the hypothesis of Proposition~\emph{\ref{P:s(0)=s(1)=2}(b)}, the spectrum of $E$ can be
computed by the formula   
$s(i) = (\Delta^2{\fam0 h}^0_{X^\prime})(i)$, $\forall \, i \geq 1$,
\emph{(}and $s(0) = 2$\emph{)}. 
\end{corollary}

\begin{proof}
Consider the exact sequences$\, :$
\[
0 \ra \sco_\piii(-1) \ra E \ra \sci_Y(1) \ra 0\, ,\
0 \ra \sci_Y \ra \sci_{X^\prime} \ra \omega_X(2) \ra 0\, . 
\]
Since $\tH^1(\sci_{X^\prime}) = 0$ it follows, as in the proof of Lemma~\ref{L:Delta2h0},
that $s(i) = (\Delta^2\h^0_X)(i)$, $\forall \, i \geq 0$. Using the exact sequence
$0 \ra \sco_L \ra \sco_X \ra \sco_{X^\prime} \ra 0$ and the fact that
$\tH^0(\sco_{X^\prime}(-2)) = 0$, one gets that $\h^0_X = \h^0_{X^\prime} + \h^0_L$. 
\end{proof}

\begin{corollary}\label{C:s(0)=s(1)=s(2)=2}
Under the hypothesis of Proposition~\emph{\ref{P:s(0)=s(1)=2}(b)}, assume that, moreover,
$s(2) = 2$. Then ${\fam0 H}^0(\sci_{X^\prime}(2)) \neq 0$ and
${\fam0 H}^1(\sci_{X^\prime}(2)) = 0$. 
\end{corollary}

\begin{proof}
Since $s(2) = 2$, Theorem~\ref{T:rholeqs-1}(c) implies that $\rho(1) = 0$.
Corollary~\ref{C:si=2}(b) implies, now, that $b_1 = \rho(-1) = 1$. One can use, finally, the
last two exact sequences in the proof of Proposition~\ref{P:s(0)=s(1)=2} and the fact that
$\rho(0) = \rho(1) = 0$. 
\end{proof}

\begin{remark}\label{R:s(0)=s(1)=2}
Under the hypothesis of Proposition~\ref{P:s(0)=s(1)=2}(b), if $L$ is not a component of
$X^\prime$ then $L \cap X^\prime = \emptyset$. Indeed, in this case one can write
$Y = Y^\prime \cup L$, where $Y^\prime$ is a locally Cohen-Macaulay curve with the same support
as $X^\prime$. One has, by Lemma~\ref{L:z0cupz1},
$\omega_Y \vert_L \simeq \sch om_{\sco_L}(\sci_{L \cap Y^\prime , L} , \omega_L)$. Since
$\omega_Y \simeq \sco_Y(-2)$ and $\omega_L \simeq \sco_L(-2)$ it follows that
$L \cap Y^\prime = \emptyset$.

In this case, the epimorphism $\sci_{X^\prime} \ra \omega_X(2)$ induces an epimorphism
$\sci_{X^\prime} \ra \omega_{X^\prime}(2)$ whose kernel is the ideal sheaf $\sci_{Y^\prime}$ of a
curve $Y^\prime$ with $\omega_{Y^\prime} \simeq \sco_{Y^\prime}(-2)$ (cf. Ferrand \cite{f}). 
Using Serre's method of extensions, one can construct an extension$\, :$
\[
0 \lra \sco_\piii(-1) \lra E^\prim \lra \sci_{Y^\prime}(1) \lra 0\, , 
\]
with $E^\prim$ locally free. $E^\prim$ has Chern classes $c_1^\prime = 0$,
$c_2^\prime = c_2 - 1$, is stable and its spectrum $s^\prime$ has the property that
$s^\prime(0) = 1$ and $s^\prime(i) = s(i)$ for $i \geq 1$.

\vskip2mm 

{\bf Question :} Is there a stable vector bundle $E^\prim$ with the above spectrum even in the
case where $L$ is a component of $X^\prime$ ?
\end{remark}

\appendix
\section{Curves in a double plane}\label{A:xin2h}

We recall, in this appendix, the results of Hartshorne and Schlesinger \cite{has} and of
Chiarli, Greco and Nagel \cite{cgn} about curves in a double plane. We shall use, several
times, the following general fact$\, :$ if $\scf$ is a sheaf on $\p^n$
such that $\text{depth}\, \scf_x \geq 1$, $\forall \, x \in \text{Supp}\, \scf$, then
any irreducible component of $\text{Supp}\, \scf$ is positive dimensional. Indeed, if
$\text{Supp}\, \scf$ would have an irreducible component consisting of a single point
$x$ then $\scf_x$ would be an Artinian $\sco_{\p^n , x}$-module hence one would have
$\text{depth}\, \scf_x = 0$. 

\begin{remark}\label{R:has}
We begin by recalling the description of curves in a double plane given by Hartshorne and
Schlesinger. Let $H \subset \piii$ be a plane of equation $h = 0$. Choose a polynomial
subalgebra $R$ of $S = k[T_0 , \ldots , T_3]$ such that the composite map
$R \hookrightarrow S \ra S/Sh$ is bijective. Let $X$ be a space curve (recall that this means
purely 1-dimensional locally Cohen-Macaulay closed subscheme of $\piii$) contained
(as a subscheme) in the effective divisor $2H$. One associates to $X$ two plane curves
$C_0 \subset H$, $C \subset H$ and a 0-dimensional (or empty) subscheme $\Gamma$ of $H$
by the relations $\sci_{C_0} := (\sci_X : h)$ and
$\sci_{X \cap H , H} = \sci_{\Gamma , H}\sci_{C , H}$. Let $h = f_0 = 0$ (resp., $h = f = 0$) be
the equations of $C_0$ (resp., $C$), where $f_0$ (resp., $f$) is a homogeneous element
of $R$ of degree $r_0$ (resp., $r$). One has$\, :$
\[
\sci_X \cap \sci_H = \sci_{C_0}\sci_H \, ,\  
\sci_{X \cap H , H} := (\sci_X + \sci_H)/\sci_H \simeq \sci_X/(\sci_X \cap \sci_H) =
\sci_X/(\sci_{C_0}\sci_H). 
\]
Since $\sci_X + \sci_H \subseteq \sci_{C_0}$ it follows that
$\sci_{\Gamma , H}\sci_{C , H} \subseteq \sci_{C_0 , H}$ hence $C_0 \subseteq C$ (as schemes).
Since $\sci_{C , H} = \sci_C/\sci_H$ one gets that
$\sci_C/(\sci_X + \sci_H) \simeq \sco_\Gamma(-r)$. Using, now, the exact sequence$\, :$
\[
0 \lra (\sci_X + \sci_H)/\sci_X \lra \sci_C/\sci_X \lra \sci_C/(\sci_X + \sci_H)
\lra 0 
\]
and the fact that $\sci_{C_0}$ annihilates $(\sci_X + \sci_H)/\sci_X$ one deduces that
$\sci_{C_0}$ annihilates $\sci_C/\sci_X$ (by the general fact recalled at the beginning
of the appendix) hence $\sci_{C_0}\sci_C \subseteq \sci_X$.
One also deduces that $\sci_{C_0}$ annihilates
$\sci_C/(\sci_X + \sci_H) \simeq \sco_\Gamma(-r)$ hence $\Gamma \subset C_0$ (as
schemes). Since$\, :$
\[
\sci_{C_0}\sci_C/(\sci_{C_0}\sci_H) = \sci_{C_0}\sci_C/(\sci_{C_0}\sci_C \cap \sci_H)
\simeq (\sci_{C_0}\sci_C + \sci_H)/\sci_H 
\]
it follows that
$\sci_X/(\sci_{C_0}\sci_C) \simeq \sci_{\Gamma , H}\sci_{C , H}/(\sci_{C_0 , H}\sci_{C , H})
\simeq \sci_{\Gamma , C_0}(-r)$.

\vskip2mm

\noindent
{\bf Claim.}\quad
$\sci_C/\sci_X \simeq \sch om_{\sco_{C_0}}(\sci_{\Gamma , C_0} , \sco_{C_0}(-1))$.

\vskip2mm

\noindent
\emph{Proof of the claim.} Consider the exact sequence$\, :$ 
\[
0 \lra \sci_X/(\sci_{C_0}\sci_C) \lra \sci_C/(\sci_{C_0}\sci_C) \lra \sci_C/\sci_X \lra 0 
\]
and put $\scl := \sci_X/(\sci_{C_0}\sci_C) \simeq \sci_{\Gamma , C_0}(-r)$ and
$F := \sci_C/(\sci_{C_0}\sci_C) \simeq \sco_{C_0}(-1) \oplus \sco_{C_0}(-r)$. The composite
morphism$\, :$
\[
\scl \lra F \Izo \sch om_{\sco_{C_0}}(F , {\textstyle \bigwedge}^2F) \lra
\sch om_{\sco_{C_0}}(\scl , {\textstyle \bigwedge}^2F)
\]
is the zero morphism because it is so on $C_0 \setminus \Gamma$. Using the isomorphism
of ``d\'{e}calage''
\[
\sce xt_{\sco_{C_0}}^1(\sci_C/\sci_X , \omega_{C_0}) \simeq
\sce xt_{\sco_{C_0}}^2(\sci_C/\sci_X , \omega_H) 
\]
and the fact that $\text{depth} (\sci_C/\sci_X)_x \geq 1$,
$\forall \, x \in \text{Supp} (\sci_C/\sci_X)$, (because $\sci_C/\sci_X \subset \sco_X$) one
gets that $\sce xt_{\sco_{C_0}}^1(\sci_C/\sci_X , \omega_{C_0}) = 0$ hence the morphism
\[
\sch om_{\sco_{C_0}}(F , {\textstyle \bigwedge}^2F) \lra
\sch om_{\sco_{C_0}}(\scl , {\textstyle \bigwedge}^2F) 
\]
is an epimorphism. One deduces an epimorphism
$\sci_C/\sci_X \ra \sch om_{\sco_{C_0}}(\scl , \bigwedge^2F)$ which is an isomorphism on
$C_0 \setminus \Gamma$ hence it is an izomorphism on $C_0$. \hfill $\Diamond$ 
\end{remark}

\begin{remark}\label{R:homegax}
We want to describe the kernel and the image of the multiplication by
$h \colon \omega_X \ra \omega_X(1)$, where $X$ is as in Remark~\ref{R:has}.

Multiplication by $h \colon \sco_X(-1) \ra \sco_X$ factorizes as
$\sco_X(-1) \twoheadrightarrow \sco_{C_0}(-1) \overset{h}{\ra} \sco_X$. Applying
$\sch om_{\sco_X}(\ast , \omega_X)$, one deduces that the multiplication by
$h \colon \omega_X \ra \omega_X(1)$ factorizes as
$\omega_X \ra \omega_{C_0}(1) \hookrightarrow \omega_X(1)$. Consider, now, the exact
sequences$\, :$
\[
0 \ra \sco_{C_0}(-1) \xra{h} \sco_X \ra \sco_{X \cap H} \ra 0\, ,\
0 \ra \sci_C/(\sci_X + \sci_H) \ra \sco_{X \cap H} \ra \sco_C \ra 0\, , 
\]
and recall that $\sci_C/(\sci_X + \sci_H) \simeq \sco_\Gamma(-r)$. Applying
$\sch om_{\sco_X}(\ast , \omega_X)$ to the latter exact sequence one gets that
$\omega_C \izo \omega_{X \cap H}$ and that
$\sce xt_{\sco_X}^1(\sco_{X \cap H} , \omega_X) \izo \omega_\Gamma(r)$. 
Applying it, now, to the former exact sequence one
gets an exact sequence$\, :$
\[
0 \lra \omega_C \lra \omega_X \lra \omega_{C_0}(1) \lra \omega_\Gamma(r) \lra 0
\]
from which one deduces that the kernel of the multiplication by
$h \colon \omega_X \ra \omega_X(1)$ is $\omega_C$ and its image is
$\sci_{\Gamma , C_0} \otimes \omega_{C_0}(1)$. As a byproduct, one also gets that $\Gamma$
is a Gorenstein scheme, hence it is locally complete intersection in $H$. 
\end{remark}

\begin{remark}\label{R:cgn}
We recall here the results of Chiarli et al. \cite{cgn} about the equations of a curve
in a double plane. Using the notation from Remark~\ref{R:has}, consider the diagram$\, :$
\[
\SelectTips{cm}{12}\xymatrix{ & &
\sci_{X \cap H , H}\ar[r]\ar[d]^{\mu}\ar[rd]^-{\iota}\ar @{-->}[ld]_-{\theta} &
\sci_{C , H}\ar[d] & \\
0\ar[r] & \sci_H/(\sci_{C_0}\sci_H)\ar[r]^-u & \sci_C/(\sci_{C_0}\sci_C)\ar[r]^-q &
\sci_C/(\sci_{C_0}\sci_C + \sci_H)\ar[r] & 0} 
\]
where $\mu$ is the obvious morphism
$\sci_{X \cap H , H} \simeq \sci_X/(\sci_{C_0}\sci_H) \ra \sci_C/(\sci_{C_0}\sci_C)$. Since
$\sci_C/(\sci_{C_0}\sci_C + \sci_H) \simeq \sco_{C_0}(-r)$, $q$ has a right inverse $v$
mapping the class of $f \pmod{\sci_{C_0}\sci_C + \sci_H}$ to the class of
$f \pmod{\sci_{C_0}\sci_C}$. Then there is a unique morphism
$\theta \colon \sci_{X \cap H , H} \ra \sci_H/(\sci_{C_0}\sci_H)$ such that
$\mu - v \circ \iota = u \circ \theta$.

\vskip2mm

\noindent
{\bf Claim.}\quad
$\sce xt_{\sco_H}^1(\theta , \text{id}) \colon \sce xt_{\sco_H}^1(\sci_H/(\sci_{C_0}\sci_H) , \sco_H)
\lra \sce xt_{\sco_H}^1(\sci_{X \cap H , H} , \sco_H)$
\emph{is an epimorphism}.

\vskip2mm

\noindent
\emph{Proof of the claim.} $\mu$ factorizes through the inclusion
\[
{\overline \mu} \colon \sci_X/(\sci_{C_0}\sci_C) \lra \sci_C/(\sci_{C_0}\sci_C) 
\]
and so do the
morphisms $\iota$ and $\theta$ inducing morphisms
\[
{\overline \iota} \colon \sci_X/(\sci_{C_0}\sci_C) \ra \sci_C/(\sci_{C_0}\sci_C + \sci_H)
\text{ and } 
{\overline \theta} \colon \sci_X/(\sci_{C_0}\sci_C) \ra \sci_H/(\sci_{C_0}\sci_H),
\]
respectively. Since
$\sce xt_{\sco_H}^2(\sci_C/\sci_X , \sco_H) = 0$ it follows that$\, :$
\[
\sce xt_{\sco_H}^1({\overline \mu} , \text{id}) \colon
\sce xt_{\sco_H}^1(\sci_C/(\sci_{C_0}\sci_C) , \sco_H) \lra
\sce xt_{\sco_H}^1(\sci_X/(\sci_{C_0}\sci_C) , \sco_H) 
\]
is an epimorphism. Using the commutative diagram$\, :$
\[
\SelectTips{cm}{12}\xymatrix{0\ar[r] & \sci_{X \cap H , H}\ar[r]\ar[d] & \sci_{C , H}\ar[r]\ar[d] &
\sci_C/(\sci_X + \sci_H)\ar[r] \ar @{=}[d] & 0\\
0\ar[r] & \sci_X/(\sci_{C_0}\sci_C)\ar[r]^-{{\overline \iota}} &
\sci_C/(\sci_{C_0}\sci_C + \sci_H)\ar[r] &
\sci_C/(\sci_X + \sci_H)\ar[r] & 0}
\]
one gets an exact sequence$\, :$
\begin{gather*}
\sce xt_{\sco_H}^1(\sci_C/(\sci_{C_0}\sci_C + \sci_H) , \sco_H)
\xra{\sce xt^1({\overline \iota} , \text{id})}
\sce xt_{\sco_H}^1(\sci_X/(\sci_{C_0}\sci_C) , \sco_H) \lra\\ 
\lra \sce xt_{\sco_H}^1(\sci_{X \cap H , H} , \sco_H) \lra 0\, . 
\end{gather*}
Since, as we saw above,  $\sce xt_{\sco_H}^1({\overline \mu} , \text{id})$ is an epimorphism one 
deduces, using the last exact sequence, that $\sce xt_{\sco_H}^1(\theta , \text{id})$ is an
epimorphism. \hfill $\Diamond$

\vskip2mm

It follows, from the claim, that $\sce xt_{\sco_H}^1(\sci_{X \cap H , H} , \sco_H)$ is
generated, locally, by a single element, hence $\Gamma$ is locally complete intersection
in $H$.

Consider, now, a resolution$\, :$
\[
0 \lra {\textstyle \bigoplus}_{j=1}^m\sco_H(-b_j) \overset{A^{\text{t}}}{\lra} 
{\textstyle \bigoplus}_{i=0}^m\sco_H(-a_i)
\xra{(g_0 , \ldots , g_m)} \sci_{\Gamma , H} \lra 0\, , 
\]
where $g_i \in R_{a_i}$ and where $A^{\text{t}}$ is the transpose of an $m \times (m+1)$
matrix $A$ with homogeneous entries in $R$. $\theta$ can be lifted to a morphism of
resolutions$\, :$
\[
\SelectTips{cm}{12}\xymatrix @C=10mm{0\ar[r] &
{\textstyle \bigoplus}_{j=1}^m\sco_H(-b_j-r)\ar[r]^-{A^{\text{t}}}\ar[d]^{(F_1 , \ldots , F_m)} &
{\textstyle \bigoplus}_{i=0}^m\sco_H(-a_i-r)\ar[r]^-{(fg_0 , \ldots , fg_m)}\ar[d]^{(G_0 , \ldots , G_m)}
& \sci_{X \cap H , H}\ar[r]\ar[d]^\theta & 0\\
0\ar[r] & \sco_H(-r_0-1)\ar[r]^-{f_0} & \sco_H(-1)\ar[r]^-{{\widehat h}} &
\sci_H/(\sci_{C_0}\sci_H)\ar[r] & 0}
\]
with $G_i$, $F_j$ homogeneous elements of $R$. The above claim implies that the morphism$\, $
\[
\begin{bmatrix}
  & F_1\\
A & \vdots\\
  & F_m   
\end{bmatrix}
\colon \begin{matrix} {\textstyle \bigoplus}_{i=0}^m\sco_H(a_i+r)\\ \oplus\\ \sco_H(r_0+1)
\end{matrix} \lra {\textstyle \bigoplus}_{j=1}^m\sco_H(b_j+r)  
\]
is an epimorphism. Since the ideal sheaf $\sci_{C_0}\sci_C$ defines an arithmetically
Cohen-Macaulay curve in $\piii$ one deduces easily, using the diagram from the beginning of
the remark, that the homogeneous ideal $I(X) \subset S$ is generated by$\, :$
\[
h^2\, ,\  hf_0\, ,\  f_0f\, ,\  hG_i + fg_i\, ,\  i = 0 , \ldots , m\, . 
\]
Actually, $f_0f$ is not a minimal generator of $I(X)$ because it can be expressed as a
combination of the other generators from the above system (one can write
$f_0 = \alpha_0g_0 + \cdots + \alpha_mg_m$ and $\theta(f_0f) = 0$ hence
$\alpha_0G_0 + \cdots + \alpha_mG_m = f_0F$, for some homogeneous $F \in R$). 
\end{remark}

\section{Stable 2-bundles with $c_1 = 0$ and ${\fam0 h}^0(E(1)) = 2$}\label{A:h0e(1)=2}

We include, in this appendix, a proof of Lemma~\ref{L:h0e(1)=2}. We begin by recalling a
well known fact.

\begin{lemma}\label{L:z0cupz1}
Let $Z$ be a purely $1$-dimensional locally Cohen-Macaulay projective scheme and let
$Z_0$ and $Z_1$ be purely $1$-dimensional closed subsets of $Z$, with no common irreducible
component, such that $Z = Z_0 \cup Z_1$ as sets. Endow $Z_i$ with the structure of closed
subscheme of $Z$ defined by the ideal sheaf
${\fam0 Ker}(\sco_Z \ra u_{i \ast}\sco_{Z \setminus Z_{1-i}})$, where
$u_i \colon Z \setminus Z_{1-i} \ra Z$ is the inclusion morphism, $i = 0,\, 1$. Then$\, :$

\emph{(a)} $Z_0$ and $Z_1$ are locally Cohen-Macaulay and $Z = Z_0 \cup Z_1$ as schemes.

\emph{(b)} Denoting by $D$ the scheme $Z_0 \cap Z_1$, if $Z$ is Gorenstein at every point
of $D$ then $\omega_Z \vert_{Z_i} \izo \sch om_{\sco_{Z_i}}(\sci_{D , Z_i} , \omega_{Z_i})$,
$i = 0,\, 1$. 
\end{lemma}

\begin{proof}
(a) $Z_i$ is locally Cohen-Macaulay because its structure sheaf embeds into
$u_{i \ast}\sco_{Z \setminus Z_{1-i}}$, $i = 0,\, 1$, and $Z = Z_0 \cup Z_1$ as schemes because
the morphism $\sco_Z \ra \bigoplus_{i=0}^1u_{i \ast}\sco_{Z \setminus Z_{1-i}}$ is a monomorphism. 

(b) Applying $\sch om_{\sco_Z}(\ast , \omega_Z)$ to the exact sequence$\, :$
\[
0 \lra \sci_{Z_{1-i}}/\sci_Z \lra \sco_Z \lra \sco_{Z_{1-i}} \lra 0 
\]
and using the fact that
\[
\sci_{Z_{1-i}}/\sci_Z = \sci_{Z_{1-i}}/(\sci_{Z_{1-i}} \cap \sci_{Z_i}) \simeq
(\sci_{Z_{1-i}} + \sci_{Z_i})/\sci_{Z_i} = \sci_D/\sci_{Z_i} 
\]
one gets an exact sequence$\, :$
\[
0 \lra \omega_{Z_{1-i}} \lra \omega_Z \lra \sch om_{\sco_{Z_i}}(\sci_{D , Z_i} , \omega_{Z_i})
\lra 0\, . 
\]
The hypothesis implies that $(\omega_Z \vert_{Z_i})_x$ has depth $\geq 1$,
$\forall \, x \in Z_i$. The epimorphism
$\omega_Z \vert_{Z_i} \ra \sch om_{\sco_{Z_i}}(\sci_{D , Z_i} , \omega_{Z_i})$ deduced from the above
exact sequence is an isomorphism on $Z_i \setminus D$ hence it is an isomorphism on $Z_i$. 
\end{proof}

\begin{proof}[Proof of Lemma~\emph{\ref{L:h0e(1)=2}}]
Assume that $\h^0(E(1)) \geq 2$. Let $s$ be an arbitrary non-zero global section of $E(1)$
and let $X$ be its zero scheme. $X$ is a locally complete intersection curve in $\piii$ and
one has an exact sequence$\, :$
\[
0 \lra \sco_\piii(-1) \overset{s}{\lra} E \lra \sci_X(1) \lra 0\, ,
\]
from which one deduces that $\omega_X \simeq \sco_X(-2)$. Since $\h^0(E(1)) \geq 2$ it
follows that $\tH^0(\sci_X(2)) \neq 0$, i.e., $X$ is a subscheme of some effective divisor
$\Sigma$ on $\piii$ of degree 2.

\vskip2mm

\noindent
{\bf Case 1.}\quad $\Sigma$ \emph{is a nonsingular quadric surface}.

\vskip2mm

In this case, $X$ must be an effective divisor of type $(a , b)$ on
$\Sigma \simeq \pj \times \pj$. Since $\omega_X \simeq \sco_\Sigma(a-2 , b-2) \vert_X$ it
follows that $\sco_\Sigma(a , b) \vert_X \simeq \sco_X$. Using the intersection pairing on
$\Sigma$ one deduces that $2ab = 0$ hence $a = 0$ or $b = 0$.

\vskip2mm

\noindent
{\bf Case 2.}\quad $\Sigma$ \emph{is a quadric cone}.

\vskip2mm

This case cannot occur. Indeed, any curve on a quadric cone is arithmetically Cohen-Macaulay
hence $\tH^1_\ast(\sci_X) = 0$. This would imply that $\tH^1_\ast(E) = 0$ hence that $E$ is a
direct sum of line bundles, which is not the case.

\vskip2mm

\noindent
{\bf Case 3.}\quad $\Sigma$ \emph{is the union of two distinct planes}.

\vskip2mm

This case cannot occur, either. Indeed, assume that $\Sigma = H_0 \cup H_1$ and put
$L : = H_0 \cap H_1$. Assume, firstly, that $L$ is a component of $X$. Then, as we saw in
Remark~\ref{R:xinh0h1}, either $X$ is arithmetically Cohen-Macaulay which, as we saw while 
treating Case 2, is not possible or $X$ contains an arithmetically Cohen-Macauly
curve $X^\prime$ such that $\sci_{X^\prime}/\sci_X \simeq \sco_L(m)$, for some integer $m$.
Applying $\sch om_{\sco_X}(\ast , \omega_X)$ to the exact sequence
$0 \ra \sco_L(m) \ra \sco_X \ra \sco_{X^\prime} \ra 0$ and recalling that
$\omega_X \simeq \sco_X(-2)$ one gets that $m = 0$. Using, now, the exact sequence
$0 \ra \sci_X \ra \sci_{X^\prime} \ra \sco_L \ra 0$ one gets that
$\h^1(E(-1)) = \h^1(\sci_X) = 1$. But, by Riemann-Roch, $\h^1(E(-1)) - \h^2(E(-1)) = c_2$,
hence the fact that $\h^1(E(-1)) = 1$ \emph{contradicts} our assumption that $c_2 \geq 2$.

Consequently, $L$ cannot be a component of $X$. In this case $X = C_0 \cup C_1$ with
$C_i \subset H_i$ such that $L$ is not a component of $C_i$, $i = 0,\, 1$. The scheme
$D : = C_0 \cap C_1$ is 0-dimensional. By Lemma~\ref{L:z0cupz1},
$\omega_X \vert_{C_i} \simeq \sch om_{\sco_{C_i}}(\sci_{D , C_i} , \omega_{C_i})$. Since
$\omega_X \simeq \sco_X(-2)$ one deduces that either $C_i = \emptyset$ or $C_i$ is a line,
$i = 0,\, 1$, and $C_0 \cap C_1 = \emptyset$. It follows that $\text{deg}\, X \leq 2$,
hence $c_2 \leq 1$. This \emph{contradiction} shows, finally, that Case 3 cannot occur.

\vskip2mm

\noindent
{\bf Case 4.}\quad $\Sigma$ \emph{is a double plane}.

\vskip2mm

We use, in this case, the notation from Remark~\ref{R:has}. As we saw in
Remark~\ref{R:homegax}, there is an exact sequence$\, :$
\[
0 \lra \omega_C \lra \omega_X \lra \sci_{\Gamma , C_0} \otimes \omega_{C_0}(1) \lra 0\, . 
\]
Since $\omega_X \simeq \sco_X(-2)$ it follows that
$\sci_{\Gamma , C_0} \simeq \omega_{C_0}^{-1}(-3) \simeq \sco_{C_0}(-r_0)$. As we noticed at the
end of Remark~\ref{R:has},
$\sci_C/\sci_X \simeq \sch om_{\sco_{C_0}}(\sci_{\Gamma , C_0} , \sco_{C_0}(-1)) \simeq \omega_{C_0}(2)$.
It thus remains to show that $C = C_0$. One has, anyway, $C = C_0 + C_1$, for some effective 
divisor $C_1$ on $H$, of equations $h = f_1 = 0$, for some $f_1 \in R$ (such that $f = f_0f_1$).  
Assume, by contradiction, that $C_1 \neq \emptyset$.

\vskip2mm

\noindent
{\bf Claim.}\quad $C_0 \cap C_1$ \emph{is} 0-\emph{dimensional}.

\vskip2mm

\noindent
\emph{Proof of the claim.} Assume, by contradiction, that $C_0$ and $C_1$ have a common
irreducible component $D$. We shall use the results of Chiarli et al. \cite{cgn} recalled
in Remark~\ref{R:cgn}. Since $\sci_{\Gamma , C_0} \simeq \sco_{C_0}(-r_0)$, $\Gamma$ is the
complete intersection, in $H$, of $C_0$ and of another effective divisor $C_0^\prime$ on $H$
of degree $r_0$. Assume that $C_0^\prime$ has equations $h = f_0^\prime = 0$, with
$f_0^\prime \in R$. Tensorizing by $\sco_H(r)$ the last diagram in Remark~\ref{R:cgn},
one gets a commutative diagram$\, :$
\[
\SelectTips{cm}{12}\xymatrix{0\ar[r] &
\sco_H(-2r_0)\ar[r]^-{(-f_0^\prime , f_0)^{\text{t}}}\ar[d]^{F_1} &
2\sco_H(-r_0)\ar[r]^-{(f_0 , f_0^\prime)}\ar[d]^{(G_0 , G_1)} &
\sci_{\Gamma , H}\ar[r]\ar[d]^\theta & 0\\
0\ar[r] & \sco_H(r-r_0-1)\ar[r]^-{f_0} & \sco_H(r-1)\ar[r] & \sco_{C_0}(r-1)\ar[r] & 0}
\]
By the observation following the last diagram in Remark~\ref{R:cgn}, $-f_0^\prime$, $f_0$
and $F_1$ have no common zero on $H$, i.e., $F_1$ vanishes at no point of $\Gamma$. One has$\, :$
\[
f_0F_1 = -f_0^\prime G_0 + f_0G_1 \text{ hence } f_0(F_1-G_1) = -f_0^\prime G_0\, . 
\]
One deduces that $f_0$ divides $G_0$ in $R$, i.e., that $G_0 = G_0^\prime f_0$, for some
$G_0^\prime \in R$. Consequently, $G_1 = F_1 + G_0^\prime f_0^\prime$. In particular, $G_1$
vanishes at no point of $\Gamma$. 

By the result of Chiarli et al. recalled at the end of Remark~\ref{R:cgn},
$I(X) \subset S$ is generated by$\, :$
\[
h^2\, ,\  hf_0\, ,\  hG_0 + f_0f\, ,\  hG_1 + f_0^\prime f\, . 
\]
Now, notice the relation$\, :$
\[
f_0^\prime \cdot (hG_0 + f_0f) = f_0 \cdot (hG_1 + f_0^\prime f) - F_1 \cdot hf_0\, . 
\]
This relation shows that, locally, arround a point $x \in C_0 \setminus \Gamma$, $\sci_{X , x}$ is
generated only by the elements$\, :$
\[
h^2\, ,\  -hf_0\, ,\  hG_1 + f_0^\prime f\, . 
\]
One gets, locally, a free resolution$\, :$
\[
0 \ra 2\sco_{\piii , x}
\xra{\left(\begin{smallmatrix} f_0 & h & 0\\ G_1 & -f_0^\prime f_1 & -h
\end{smallmatrix}\right)^{\text{t}}} 3\sco_{\piii , x} \lra \sci_{X , x} \lra 0\, . 
\]
If, as we supposed, $C_0$ and $C_1$ have a common irreducible component $D$ then there is a
point $x \in D$ such that $G_1$ vanishes at $x$. This point $x$ does
not belong to $\Gamma$ because $G_1$ vanishes at no point of $\Gamma$. Then the above free
resolution is minimal hence $X$ is not locally complete intersection at $x$ and this is a
\emph{contradiction}. \hfill $\Diamond$

\vskip2mm

Since $C_0$ and $C_1$ have no common irreducible component, one must have, according to
Lemma~\ref{L:z0cupz1}(a), $X = X_0 \cup C_1$, where $X_0$ is a locally Cohen-Macaulay curve
with the same underlying set as $C_0$. Then Lemma~\ref{L:z0cupz1}(b) implies that
$\omega_X \vert_{C_1} \simeq \sch om_{\sco_{C_1}}(\sci_{X_0 \cap C_1 , C_1} , \omega_{C_1})$. Since
$X_0 \cap C_1 \neq \emptyset$ the last isomorphism clearly \emph{contradicts} the fact that
$\omega_X \simeq \sco_X(-2)$. This contradiction shows that $C_1 = \emptyset$. 
\end{proof}

\end{document}